\documentclass[11pt]{amsart}
\usepackage{amssymb}
\usepackage{amsmath} 
\usepackage{amsthm} 
\usepackage{epsfig}
\usepackage{graphicx}
\usepackage{color}
\usepackage{transparent}
\usepackage{soul}
\usepackage{subfig}
\usepackage{wrapfig}
\usepackage[colorinlistoftodos]{todonotes}

\usepackage{tikz-cd}

\numberwithin{equation}{section}

\newcommand{\ve}{\mathbf{e}}
\newcommand{\ie}{\mathit{e}}

\newcommand{\cE}{\mathcal{E}}

\newcommand{\p}{\mathbf{P}}
\newcommand{\zm}{\mathbb{Z}^{1,n}}
\newcommand{\pos}{\Phi^+_n}
\newcommand{\nee}{\Phi^-_n}

\newtheorem*{thma}{Theorem A}

\newtheorem*{thmb}{Theorem B}

\newtheorem*{thmc}{Theorem C}

\newtheorem{thm}{Theorem}[section]
\newtheorem{lem}[thm]{Lemma}
\newtheorem{defn}[thm]{Definition}
\newtheorem{cor}[thm]{Corollary}

\newtheorem{prop}[thm]{Proposition}

\title{The Dynamical Degrees of Rational Surface Automorphisms}

\author{Kyounghee Kim}
\address{Department of Mathematics\\
         Florida State University\\
         Tallahassee, FL 32306}
\email{kim@math.fsu.edu}


\begin{document}

\maketitle

\begin{abstract}
The induced action on the Picard group of a rational surface automorphism with positive entropy can be identified with an element of the Coxeter group associated to $E_n, n\ge 10$ diagram. It follows that the set of dynamical degrees of rational surface automorphisms is a subset of the spectral radii of elements in the Coxeter group. This article concerns the realizability of an element of the Coxeter group as an automorphism on a rational surface with an irreducible reduced anti-canonical curve. For any unrealizable element, we explicitly construct a realizable element with the same spectral radius. Hence, we show that the set of dynamical degrees and the set of spectral radii of the Coxeter group are, in fact, identical. This has been shown by Uehara in \cite{Uehara:2010} by explicitly constructing a rational surface automorphism. This construction depends on a decomposition of an element of the Coxeter group. Our proof is conceptual and provides a simple description of elements of the Coxeter group, which are realized by automorphisms on anti-canonical rational surfaces.\end{abstract}

\section{Introduction}\label{S:intro}
A compact complex surface $S$ admitting an automorphism is rare. However, it is shown \cite{Bedford-Kim:2006, McMullen:2007, Bedford-Kim:2010} that there are infinite families of rational surfaces admitting automorphisms. Several methods for constructing rational surfaces and studying the dynamics of automorphisms on them have been considered, and the dynamical properties of such automorphisms have been studied in several recent papers (for example, \cite{Bedford-Kim:rotation, Blanc:2008, Cantat-Dolgachev, Lesieutre:2021}).

Suppose $S$ is a rational surface obtained by blowing up a set $P$ of $n$ (possibly infinitely near) points in $\p^2$. The Picard group $Pic(S)$  is generated by the class $\ve_0$ of a strict transform of a generic line in $\p^2$ and the classes $\ve_i$ of total transforms of exceptional curves over points in $P$. There is a natural isomorphism $\phi$ between $\mathbb{Z}^{1,n}$ with a Minkowski inner product and a Picard group $Pic(S)$. By Nagata \cite{Nagata, Nagata2}, if $F:S \to S$ is an automorphism with infinite order, then there is an element $\omega$ of the Weyl group $W_n$ generated by a set of reflections acting on $\mathbb{Z}^{1,n}$ such that the induced action $F_* : Pic(S) \to Pic(S)$ is equivalent to $\omega$ via the isomorphism $\phi$, in other words, the following diagram commutes:
\begin{center}
\begin{tikzcd}
\mathbb{Z}^{1,n} \arrow[r, "\omega"] \arrow[d, "\phi",labels=left ]
& \mathbb{Z}^{1,n} \arrow[d, "\phi" ] \\
Pic(S) \arrow[r,  "F_*"]
& Pic(S)
\end{tikzcd}
\end{center}
In this case, we say $F$ \textit{realizes} $\omega \in W_n$. The generators of this Weyl group $W_n$ are reflections through a set of vectors $\alpha_i \in \mathbb{Z}^{1.n}$, $i=0,\dots, n-1$ such that \[ \alpha_0 = e_0- e_1-e_2-e_3,  \qquad \text{and } \qquad \alpha_i = e_i - e_{i+1}, i =1, \dots, n-1.\]This Weyl group is isomorphic to the Coxeter group associated to the Coxeter diagram $E_n$. (See \cite{Humphreys:1990}.)

If $F$ is a rational surface automorphism, the dynamical degree $\delta(F)$ is given by 
\[ \delta(F) = \lim_{n \to \infty} (\text{algebraic degree of } F^n)^{1/n}. \] Let $\Delta$ be the set of dynamical degrees of rational surface automorphisms and $\Lambda$ be the set of spectral radii of Weyl group elements:
\[ \Delta = \{ \delta(F) : F \text{ is a rational surface automorphism.}\}, \qquad \text{ and } \]
\[ \Lambda = \{ \lambda (\omega) : \omega \in W_n \text{ for som } n\ge 1 \}\]
where $\delta$ denotes the dynamical degree and $\lambda$ denotes the spectral radius. According to Diller and Favre \cite{Diller-Favre:2001}, the dynamical degree $\delta(F)$ of an automorphism $F$ is given by the spectral radius of $F_*$. Thus, we have $\Delta \subset \Lambda$. In this article, we show $\Lambda = \Delta$.
\begin{thma}
For any $\lambda \in \Lambda$, there is an element $\omega \in W_{n\ge 1}$  such that $\omega = F_*$ for a rational surface automorphism $F$ with $\delta(F) = \lambda$. Thus we have \[ \Delta = \Lambda.\]
\end{thma}

This theorem was shown by Uehara \cite{Uehara:2010} by constructing an automorphism with a given dynamical degree as a composition of quadratic birational maps. The main idea in \cite{Uehara:2010} is identifying possible \textit{orbit data} (See Section \ref{SS:orbitdata} for the definition.) for quadratic birational maps. The brief sketch of Uehara's construction is the following: Let $\kappa$ be a reflection through $\alpha_0$, and $\kappa_I$ be a reflection through $e_0 -  e_i-e_j-e_k$ for $I \subset\{ 1, \dots, n\}$ with $|I|=3$. Also, let $\Sigma$ be a span of reflections through $\alpha_i, i\ge 1$. With these notation, we see that for any $r \in \Sigma$, we have $r^{-1} \kappa\, r = \kappa_I$ for some $I \subset \{ 1, 2, \dots, n\}$. Thus for $\omega = r_0\, \kappa\, r_1\, \kappa\, r_2\, \cdots\, \kappa\, r_m \in W_n$ with $r_i \in \Sigma, i = 1, \dots, m$ we have 
\[ \omega = r \kappa_{I_m} \cdots  \kappa_{I_1} \] where 

\[ r = r_0 \, r_1\, \cdots \,r_m \in \Sigma, \ \ \ \kappa_{I_j} = (r_j \, r_{j+1}\, \cdots \, r_m)^{-1} \kappa \,  (r_j \, r_{j+1}\, \cdots \, r_m) \quad \text{for } j=1, \dots, m\]  and $I_j=\{ j_1, j_2, j_3\}$ is a set of three distinct integers in $\{1,2, \dots, n\}$. By inspecting the above decomposition, for each $j_i$, one can obtain the smallest non-negative integers $t$ and  $s_k$ such that 
\[ \kappa_{I_{s-1}} \cdots \kappa_{I_1} \omega^t  r  \kappa_{I_m} \cdots \kappa_{I_{j+1}} e_{j_i} = e_{s_k} .\]
This allows to obtain \textit{orbit data} (See Section \ref{SS:orbitdata} for the definition.) for a quadratic birational map $f_j$ whose critical images corresponding $e_{j_1}, e_{j_2}, e_{j_3}$ for each $j=1, \dots, m$. By assuming all $f_j$ fix a given cusp cubic $C = \{ \gamma(t) \}$, we have $e_{j_i} = \gamma(p_{j,i})$ and $f_j|_{C} \,\gamma(t) = \gamma(d_j t + c_j)$ for some $p_{j,j}, d_j, c_j \in \mathbb{C}$. Diller in \cite{Diller:2011} showed that the multiplier $d_j$ and the parameters for three critical images uniquely determine (up to a linear conjugacy) a birational map properly fixing a cubic with arithmetic genus $=1$. By repeatedly applying Diller's construction, Uehara obtained a system of equations of $p_{j, i}, d_j, c_j$ to get an explicit formula for the desired birational map $f = f_m \circ \cdots \circ f_1$.

Uehara's construction required expressing $\omega \in W_n$ as a composition of generators of $W_n$. However, this is a difficult task. Furthermore, the construction depends on the orbit data, which is not uniquely determined by $\omega \in W_n$.

\vspace{1ex}
In this article, we generalize McMullen's construction in \cite{McMullen:2007} by allowing infinitely near points in the base locus with the assumption of the existence of an anti-canonical curve. The isomorphism $\phi:\mathbb{Z}^{1,n} \to Pic(S)$ together with the base points is called the \textit{marked blowup} of $S$. McMullen \cite{McMullen:2007} constructs a rational surface automorphism from a marked blowup with an anti-canonical curve provided that all base points are distinct points in $\p^2$. A brief sketch of McMullen's construction is as follows: 

\vspace{1ex}
Suppose $\omega \in W_n$. If $\phi$ and $\phi\circ \omega$ give two markings for $S$, each marking gives a projection to $\mathbf{P}^2$.  Each marking would have a different base locus, and these two markings are associated with a birational map $f$, which does not necessarily lift to automorphism. 

The most well-known example would be the case with $n=3$, and $\omega$ is the Cremona involution $\kappa$, the reflection through the vector $e_0-e_1-e_2-e_3$. Let $(S,\phi)$ be a marked blowups with $\phi(e_i) = \ve_i$ where $\ve_0$ is the class of the strict transform $\tilde H$ of generic line $H$ in $\p^2$ and $\ve_i$ is the class of the exceptional curve $\mathcal{E}_i$ over a point $p_i \in \p^2$.  Let $\pi:S \to \p^2$ be the corresponding blowup along three distinct points $\{p_1,p_2,p_3\} \in \p^2$. The marked blowup $(S, \phi\circ \kappa)$ with the projection $\pi':S \to \p^2$  satisfies $\phi( \kappa e_i) = \ve'_i$ and thus the class $\ve'_i = \ve_0 - \ve_1-\ve_2-\ve_3 + \ve_i$, $i=1,2,3$ is the class of exceptional curves over a point $q_i \in \p^2$. 
\begin{center}
\begin{tikzcd}[column sep=small]
&  S \arrow["\pi",dl,labels=above left] \arrow[dr,"\pi' "]  \\
\p^2 \arrow[rr, "f", dashed] & & \p^2
\end{tikzcd}
\end{center}
Since $\omega \ie_0 = 2 \ie_0 - \ie_1-\ie_2-\ie_3$, we have $\pi' \tilde H = Q$ where $Q$ is a conic passing through $p_1,p_2,p_3$ while $\pi \tilde H = H$. Also, the exceptional curve $\mathcal{E}_i$ over $p_i$ satisfies $\pi(\mathcal{E}_i) = p_i$, and $\pi'(\mathcal{E}_i) = L'_{j,k}$ where  $L'_{i,j}$ is the line joining $q_i,q_j$. The class of the strict transform of the line joining $p_i,p_j$, $L_{i,j}$ is given by $[L_{i,j} ]= \ve_0 - \ve_i-\ve_j = \ve'_k$.
Thus we see the birational map $f$ is quadratic and satisfies
\[ f: \pi(\mathcal{E}_i) = p_i \mapsto L'_{j,k},\qquad \text{and} \qquad L_{i,j} \mapsto q_k \qquad \{i,j,k\}= \{1,2,3\}.\]

In addition, if there is an automorphism $g$ on $\p^2$ sending one base locus to another, by composition $f\circ g$ of an automorphism $g$ and a birational map $f$, we can obtain an automorphism on $S$ by lifting $f \circ g$. 

\begin{center}
\begin{tikzcd}[column sep=small]
&  S \arrow["\pi'",dl,labels=above left] \arrow[rr, "F"] &&  S \arrow["\pi",dl,labels=above left] \arrow[dr,"\pi' "]  \\
\p^2 \arrow[rr, "g"] & &\p^2 \arrow[rr, "f", dashed] & & \p^2
\end{tikzcd}
\end{center}

To have a dynamical degree $>1$, as stated by McMullen \cite{McMullen:2002}, we need to blow up at least $10$ points in $\p^2$. Thus, if two sets of points $P, P'$ with $|P|=|P'|\ge 10$ are in general position, an automorphism on $\p^2$ sending $P$ to $P'$ is special, or in fact, might not exist. 

\vspace{2ex}
One can overcome this difficulty by imposing the condition of the existence of an anti-canonical curve $Y$. 
If a point $p$ in the base locus has no infinitely near points in the base locus, then $p$ is in the strict transform of $Y$. Suppose $p^{(i)}, i=1, \dots k$ are in the base locus satisfying a point $p^{(i)}$ is an infinitely near point of $p^{(i-1)}$ of order $1$ and no other points are infinitely near to $p^{(1)}$, then there is a successive blowups $\pi_i : S_i = \text{B}\ell_{p^{(i)}} \to S_{i-1}$ of a single point $p^{(i)}$ with $S_0 = \p^2$. It follows that each strict transform of $Y$ in $S_i$ contains $p^{(i)}$. In either case, the set of base points in $\p^2$ belongs to the projection $X=\pi(Y)$. If $X$ is a rational curve with arithmetic genus $1$, it is easy to tell whether there is such an automorphism from $P$ to $P'$.

\vspace{2ex}
We say $\omega \in W_n, n\ge 10$ is \textit{essential} if $\omega$ is not conjugate into (after relabelling $\ie_i$'s if necessary) $W_{n'}$ with $n' < n$. It is shown in \cite{Kim:2022} that there is a non-essential element $\omega \in W_{14}$ such that $\omega$ is not realizable by a rational surface automorphism. This is shown by counting periodic points using the Lefschetz fixed point formula. In fact, this $\omega$ is conjugate to the Coxeter element $\omega_{10} :=s_0 \cdots s_{10} \in W_{10}$, and it is known \cite{Bedford-Kim:2006, McMullen:2007} that $\omega_{10}$ is realized by a rational surface automorphism. In this article, we let $W_n^{ess}$ denote the set of essential elements in $W_n$ and focus on realizing essential elements in $\cup_{n\ge 10} W_n^{ess}$.

\vspace{2ex}
Let $\mathbb{C}^{1,n} = \mathbb{Z}^{1,n} \otimes \mathbb{C}$ with the complex Minkowski form. For each $v \in \mathbb{C}^{1,n}$ we defined a \textit{marked blowup} $(S^v, Y^v, \phi^v) = \text{B}\ell(X,\rho^v)$ (See Section \ref{S:realization} for the explicit construction.) by blowing up points on a cuspidal cubic. Let us denote $\Phi^+_n$ as the set of positive roots and $\Phi^-_n = - \Phi^+_n$ as the set of negative roots. Also, let $\Phi^+_v$ be a set of nodal roots in $(S^v, Y^v, \phi^v)$ and let $\Sigma_v$ be the span of all reflections through nodal roots:
\[ \Sigma_v \ = \ \text{Span} \{ s_\alpha: \alpha \in \Phi^+_v\}, \quad \]
\[  \text{ where} \quad s_\alpha (x) := x + (x \cdot  \alpha ) \alpha \qquad \text{ for } x \in \mathbb{Z}^{1,n}. \]

\begin{thmb}
Let $v \in \mathbb{C}^{1,n}$ and let $\omega$ be an essential element in $W_{n\ge 10}$ with the spectral radius $\lambda$. Suppose $v$ is a leading eigenvector for $\omega$. Then we have
\begin{enumerate}
\item $\omega$ is realized by an automorphism on $(S^v, Y^v, \phi^v)$ if and only if $\omega^{-1} (\alpha)  \in \Phi^+_n$ for all $\alpha \in \Phi^+_v.$ 
\item If $\omega$ is not realized by an automorphism on $(S^v, Y^v, \phi^v)$, there is $s \in \Sigma_v$ such that $s \circ \omega$ is realized by a rational surface automorphism on $(S^v, Y^v, \phi^v)$ and we have $\lambda(s\circ \omega)=\lambda(\omega)$.
\end{enumerate}
\end{thmb}

Also, via an example (see Theorems \ref{T:nono} and \ref{T:yes}), we show that there is an essential element $\omega \in W_{n\ge 10}$ such that $\omega$ can not be realized by any automorphisms on anti-canonical rational surfaces. Then, using Theorem B, we find a realizable $\omega'$ with the same spectral radius. The method in Theorem B can be done in finitely many steps due to Corollary \ref{C:finite}.

\begin{thmc}
There is an essential element $\omega \in  \cup_{n\ge 10} W_n$ such that $\omega$ can not be realized by an automorphism on an anti-canonical rational surface. 
\end{thmc}

This article is organized as follows: Section \ref{S:wgroup} discusses Coxeter groups and their essential elements. In Sections \ref{S:gMB} --  \ref{S:auto}, we generalize McMullen's definitions \cite{McMullen:2007} of Marked blowups and Marked cubics by allowing infinitely near points as base points. Also, we show that the set of nodal roots is finite, provided a unique anti-canonical curve exists. In Section \ref{S:realization}, we prove Theorem A and Theorem B. A couple of basic properties of quadratic birational maps and their orbit data are discussed in Section \ref{S:quad}. In this section, we also briefly explain Diller's construction in \cite{Diller:2011} and a few more cases omitted in \cite{Diller:2011}. In Section \ref{S:trouble}, using a specific example, we prove Theorem C and show how to use Theorem B. 

\subsection*{Acknowledgement}
We would like to thank Jeffery Diller for his comments and attention and the referee for suggesting a better title. This material is partly based upon work supported by the National Science Foundation under DMS-1928930 while the author participated in a program hosted by the Mathematical Sciences Research Institute in Berkeley, California during Spring 2022 semester. 

\section{Coxeter Group}\label{S:wgroup}
Let $\zm$ be a lattice of signature $(1,n)$ with the basis $( \ie_0,\dots, \ie_n)$ such that \[ \ie_0\cdot \ie_0 =1, \quad \ie_i\cdot \ie_i = -1, i=1,\dots, n,\quad \text{and} \quad \ie_i\cdot \ie_j = 0, i \ne j.\] 
 Using this basis as coordinates, the Minkowski inner product is given by \[ x\cdot x \ = \ x_0^2 - x_1^2 - x_2^2- \cdots -x_n^2\qquad \text{for} \ \ x = \sum x_i \ie_i.\]
  The canonical vector is given by $\kappa_n = (-3,1,1,\dots,1)\in \zm$. The subgroup $L_n  = \kappa_n^\perp \subset \zm$ orthogonal to $\kappa_n$ is generated by a set of vectors \[ \alpha_0 = \ie_0 - \ie_1-\ie_2-\ie_3, \quad \text{and} \quad \alpha_i = \ie_i - \ie_{i+1}, \ i=1, \dots, n-1.\]
 The generating vectors $\alpha_i, i=0,\dots, n-1$ satisfy
 \[  \kappa_n \cdot \alpha_i =0  \quad \text{and} \quad \alpha_i \cdot \alpha_i = -2,\  \text{for }0 \le i \le n-1. \]
For each $i=0, \dots, n-1$, we can define a reflection $s_i$ on $\zm$ \[ s_i (x) = x + (x\cdot \alpha_i ) \, \alpha_i .\] We see that $s_i(\kappa_n) = \kappa_n$, $s_i(\alpha_i) = - \alpha_i$, and $s_i(\alpha_j) = \alpha_j$ for $j \ne i$. The group $W_n$ generated by those reflections $s_0, \dots, s_{n-1}$ fixes $\kappa_n$ and preserves the Minkowski product, and we have 
\[ W_n = \langle s_0, s_1, \dots, s_{n-1} | (s_i s_j)^{m_{ij}} = 1 \rangle\] where 
\[
m_{ij} = \left\{ \begin{aligned} \ 1 \ \ \ \ \ & \text{if } i=j \\ \ 3\ \ \ \ \ & \text{if } \{i,j\} = \{0,3\} \text{ or } |i-j| = 1, i,j \ne 0 \\ \ 2\ \ \ \ \ &\text{otherwise} \\\end{aligned} \right.
\]
This group $W_n \subset O(L_n)$ is called the Coxeter group. Let $V_n$ be a inner product space over $\mathbb{R}$ with the basis $\{\alpha_0, \alpha_1, \dots, \alpha_{n-1}\}$. Then each $s_i \in W_n$ gives the unique reflection in $O(V_n)$ preserving $L_n$. The root system $\Phi_n = \cup_i W_n \alpha_i$ is the $W_n$-orbits in $V_n$ of generating vectors $\alpha_i$'s.

\subsection{Root System and Simple roots} Each vector in $\alpha \in \Phi_n$ is called a root, and each generating vector $\alpha_i$ is called a simple root. Each $\alpha \in \Phi_n$ is a linear combination $\sum_i c_i \alpha_i$ with either $c_i \ge 0$ for all $i$ or $c_i \le 0$ for all $i$. We say $\alpha \in \Phi_n$ is positive if $\alpha = \sum _i c_i \alpha_i$ with $c_i \ge 0$ for all $i=0, \dots, n-1$ and $\alpha$ is negative if $c_i \le 0$ for all $i$. 
Let $\Phi_n^+$ denote the set of all positive roots and $\Phi_n^-$ for the set of negative roots. 

\vspace{1ex}
Let us list a couple of useful facts about root systems (See \cite{Humphreys:1990} for details.)
\begin{itemize}
\item The group homomorphism $W_n \to GL(\zm)$ is injective.
\item $\Phi_n = \Phi_n^+ \cup  \Phi_n^-$.
\item  For all $\alpha_i$, $s_i ( \Phi_n^+ \setminus\{ \alpha_i\}) = \Phi_n^+ \setminus \{ \alpha_i\}$ 
\end{itemize}

\subsection{Length Function} For each $\omega\in W_n$, $\omega$ is a word in $s_i, i=0, \dots, n-1$ and the length $\ell(\omega)$ is given by the minimum possible length of a word in $s_i$ for $\omega$. Note that $\omega \in W_n$ can have several different words of length $\ell(\omega)$. One useful fact of the length of $\omega$ is its relation to the number of positive roots sent to negative roots \cite{Humphreys:1990}, that is, for $\omega \in W_n$, $\ell(\omega) = | \Pi(\omega)|$ where $\Pi(\omega)$ is the set of positive roots sent by $\omega$ to negative roots. 
\[ \ell(\omega) \ =\ |\{ \alpha \in \Phi^+_n: \omega(\alpha)  \in \Phi^-_n\}| \]


\subsection{Essential Elements} Let $S = \{ s_0, s_1, \dots, s_{n-1} \}$ be a set of generators of $W_n$. 
\begin{defn}
We say $\omega \in W_n$ is \textit{essential} if there is no subgroup $W' \subset W_n$ generated by a proper subset $S' \subset S$ such that \[\omega W' \omega^{-1} = \omega^{-1} W' \omega = W'.\] Let us denote $W_n^{ess}$ the set of essential elements of $W_n$.
\end{defn}
 Thus if $\omega \in W_n$ is non-essential, then $\omega$ can be identified as an element of a smaller Coxeter group generated by $S'$. If $S' \subset \{ s_1, s_2,\dots, s_{n-1} \}$, then $W '$ is finite and its Coxeter graph is given by $A_m$ for some $m\le n$ \cite[Chapter~2]{Humphreys:1990}. If $s_0 \in S'$, we may assume that $S' \subset \{ s_0, s_1, \dots, s_{n-2} \}$. Then we can treat $\omega$ as an element of $W_{n-1}$. If $\omega\in W_n$ is generated by $s_1, \dots, s_n$, then the spectral radius $\lambda(\omega)$ of $\omega=1$. Thus, from the definition, we see that 
 \begin{lem}\label{L:essentialdegree}
 If $\omega \in W_n$ with the spectral radius $\lambda(\omega)>1$, then there is $n'\le n$ such that $\omega$ is conjugate to an essential element $\omega' \in W_{n'}^{ess}$ and $\lambda(\omega) = \lambda(\omega')$.
 \end{lem} The essential elements of $W_n, n\ge10$ and their properties was considered by Krammer \cite{Krammer:2009} and Paris \cite{Paris:2007}.

 \begin{prop}\cite{Krammer:2009}\label{P:essential}
 Suppose a Coxeter group $W=W_n, n\ge 10$. Then
 \begin{itemize}
 \item every essential element is of infinite order,
 \item $\omega$ is essential if and only if $\omega^m, m\ne 0$ is essential, and
 \item if $\omega$ is essential then $\langle \omega \rangle = \{ \omega^m, m \in \mathbb{Z} \} $ is a finite index subgroup of the centralizer of $\omega$ in $W$. 
 \end{itemize}
 \end{prop}
 
 It is known \cite{McMullen:2002, Paris:2007} that the Coxeter element $\omega = s_0 s_1 \cdots s_{n-1} \in W_n$ is essential, and the product of $s_0, \dots, s_{n-1}$ in any other is conjugate to $\omega$. In fact, the Coxeter element has the smallest spectral radius among all essential elements in $W_n$. 
 \begin{thm}\cite[Theorem~1.2]{McMullen:2002} \label{T:minimumsp}
 Let $\omega \in W_n^{ess}$. Then the spectral radius $\lambda(\omega) \ge \lambda( s_0 s_1 \cdots s_{n-1} )$. 
 \end{thm}
 In this article, we will focus on essential elements. 
 \begin{lem}
 If $\omega \in W^{ess}_n$, then for each $i \ne 0$, there is $k \in \mathbb{Z}$ such that $\omega^k( e_i) \cdot e_0 \ne 0 $
 \end{lem}
\begin{proof}
Suppose there is $e_i$ such that $\omega^k( e_i) \cdot e_0 = 0 $ for all $k$. Then $\omega^k( e_i) = \sum_{j \ne 0} m_j e_j$ for some $m_j \in \mathbb{Z}$. Since $\omega^k$ preserves the bilinear form, we have \[\omega^k( e_i) \cdot \omega^k( e_i) =-  \sum_{j\ne 0} m_j^2 = -1.\] It follows that $m_j = 0$ for all but one $1 \le j\le n$.
Thus there is a subset $I \subset \{ 1,2, \dots, n\}$ such that  \[ \{ \omega^k( e_i), k \in \mathbb{Z} \} = \{ e_j, j \in I\}. \] We may assume that $I = \{k+1, 
\dots, n\} $ for some $k\le n-1$. Then there must be $m \in \mathbb{Z}$ such that $\omega^m e_j = e_j$ for all $j \in I$. Thus it is clear that $\omega^m W_k \omega^{-m} = W_k$, for some $k\le n-1$ and $\omega^m$ is not essential. Using the second statement of Proposition \ref{P:essential}, we see that $\omega$ is not essential.
\end{proof}

\subsection{Periodic roots}

Suppose $\omega \in W_n$. Let us denote $\lambda_\omega$, the spectral radius of $\omega$. 
\begin{prop}\label{P:orthogonal}
Let $\omega \in O(V_n)$. If $v \in V_n$ is an eigenvector corresponding to an eigenvalue $\lambda \ne \pm1$, then $v \cdot v =0$. 
\end{prop}

\begin{proof}
Since $\omega v = \lambda v$, we have \[ v\cdot v = \omega v \cdot \omega v = \lambda^2 (v\cdot v)\]
Since $\lambda^2 \ne 1$, we have $v \cdot v =0$.
\end{proof}

%
%
%
Suppose $\lambda_\omega>1$ and let $S$ be a subspace of $V_n$ generated by eigenvectors whose eigenvalues are Galois conjugate of $\lambda_\omega$. Then, any periodic roots are in the orthogonal complement $S^\perp$ .

\begin{prop}\label{P:periodic}
Suppose $\omega \in W_n$ and there exists $v \in V_n$ such that $ \omega v = \lambda_\omega v $ with $\lambda_\omega>1$. Then for all periodic roots $\alpha \in \Phi_n$, we have \[ \alpha\cdot v = 0. \]
Conversely, if $ \alpha\cdot v = 0$ then $\alpha$ is periodic.
\end{prop}

\begin{proof}
Suppose $\alpha$ is a periodic root with period $p\ge 1$. Since $\omega$ preserves the inner product, we have \[  \alpha\cdot v = \omega(\alpha) \cdot \omega(v) = \lambda_\omega \omega(\alpha) \cdot v = \lambda_\omega^2 \omega^2(\alpha)\cdot v = \cdots= \lambda_\omega ^p \omega^p(\alpha) \cdot v  \]
But since $\omega^p \alpha = \alpha$  and $\lambda_\omega >1$ we have 
\[  \alpha\cdot v = \lambda_\omega ^p \alpha \cdot v \quad \Rightarrow \quad \alpha\cdot v =0. \]
The converse follows from Theorem 2.7 in \cite{McMullen:2007}.
\end{proof}

\section{Marked Blowups}\label{S:gMB}
Let $\pi:S \to \p^2$ be a blowup of $\p^2$ along a finite set $P=\{p_1, \dots, p_n\}$ of (possibly infinitely near) points of $\p^2$. Then the birational morphism $\pi$ factors into a sequence of blowups of a point \[ \pi : S = S_n \xrightarrow{\pi_n} S_{n-1} \xrightarrow{\pi_{n-1}} \ \  \cdots  \ \  \xrightarrow{\pi_3}S_2 \xrightarrow{\pi_2} S_1 \xrightarrow{\pi_1} S_0 = \mathbf{P}^2 \] 
 where $\pi_i: S_i \to S_{i-1}$ is a blowup of a point $p_i \in S_{i-1}$. 

 Let $\cE_i$ be the total transformation of the exceptional curve over a point $p_i \in S_{i-1}$ and $\ve_i = [\cE_i]$ be its class in $Pic(S)$. \[ \cE_i = ( \pi_i \circ \cdots \circ \pi_n)^* p_i \qquad \text{ and } \qquad \ve_i = [\cE_i]\] Then $\ve_1, \ve_2, \dots, \ve_n$ together with $\ve_0= [H]$, the class of the pullback of a generic line in $\p^2$, form a \textit{geometric basis} for $Pic(S)$ equipped with inherited intersection pairing. \[ Pic(S) = \langle \ve_0, \ve_1, \dots, \ve_n \rangle, \qquad \text{where}\] \[ \ve_0\cdot \ve_0 =1,\ \  \ve_i\cdot \ve_i = -1, i\ge 1, \ \ \ \text{and}\ \ \ve_i \cdot \ve_j=0, i \ne j \] Thus we can identify $Pic(S)$ to a lattice $\mathbb{Z}^{1,n}$ with the Minkowski product. And we can consider the Weyl group $W_n$ as a subgroup of $GL(Pic(S))$. 

\subsection{Marked Blowup}
 
 A marked blowup was defined in \cite{McMullen:2007}. A marked blowup in \cite{McMullen:2007} only allows blowing up a set of distinct points in $\p^2$. Let us define a slightly modified version to allow blowing up infinitely near points. 
 
 \begin{defn} A Marked Blowup $(S, \phi)$ is a smooth projective surface $S$ equipped with an isomorphism $\phi: \zm \to Pic(S)$ such that 
 \begin{enumerate}
 \item The marking $\phi$ sends the Minkowski inner product to the intersection pairing,
 \item there exists a set $P$ of $n$ (possibly infinitely near) points $p_1, \dots, p_n$ such that  $p_j$ is infinitely near to $p_i$ implies $j >i$, and each $p_i$ has at most one infinitely near point of order 1,
 \item there exists a sequence of blowups $\pi:S \to \p^2$ along a set $P$, and
 \item the marking $\phi$ satisfies $\phi(e_0) = [H]$  and $\phi(e_i) = [\cE_i]$ for $i=1, \dots, n$, where $H$ is the pullback of a generic line in $\p^2$ and $\cE_i \subset S$ is the total transform of the exceptional curve over $p_i$.
 \end{enumerate}
 \end{defn}

Notice that the set $P$ of base points can be written as \[ \{ p_1, \dots p_n\} \ = \{ q^{(j)}_i, 1 \le j \le h_i, 1 \le i \le k\} \] where $ \sum_i  h_i = n$, $q^{(j+1)}_i$ is infinitely near of order one to $q^{(j)}_i$, and $q^{(1)}_1, q^{(1)}_2, \dots, q^{(1)}_k$ are distinct points in $\mathbb{P}^2$. We call $h_i$ the \textit{height} of a point $q_i^{(1)}$ in $\mathbb{P}^2$. Thus each $\cE_i$ is either an irreducible $(-1)$ curve if there is no infinitely near point to $p_i$ or a tower of $h-1$ $(-2)$ curves with $(-1)$ curve at the top if $p_i$ has an infinitely near point of order $j$ for $j=1,\dots, h-1$.
  The canonical class of $S$ is $\kappa_S = - 3 \ve_0 + \sum_{i\ge 1} \ve_i = \phi(\kappa_n)$ where $\kappa_n$ is the canonical vector in $\zm$. Thus the set of elements in $Pic(S)$ orthogonal to $\kappa_S$  corresponds to the root system $\Phi_n$ in $\zm$.
  
\subsection{Effective Irreducible and Reduced Anti-canonical divisor}
Suppose $S$ has an effective irreducible and reduced anti-canonical divisor. Let $C$ be a reduced irreducible curve such that $[C] = - \kappa_S$ and suppose \[P= \ \{ p_1, \dots p_n\} \ \ =\  \{ q^{(j)}_i, 1 \le j \le h_i, 1 \le i \le k\}.\] Then $C_i:=(\pi_i \circ \cdots \circ  \pi_n) C$ in  $S_{i-1}$ contains $p_i$  for each $1 \le i \le n$. It follows that 
\begin{lem} Suppose $S$ has an effective, irreducible, and reduced anti-canonical divisor and suppose $(S,\phi)$ is a marked blowup. Let $C$ be a reduced irreducible curve such that $[C] = -\kappa_S$. If $p_j$ is an infinitely near point to $p_i$ of order $1$, then $j>i$ and $p_j  = C_j \cap \mathcal{F}(p_i)$ where  $\mathcal{F}(p_i)$ is the strict transform of the exceptional curve over a point $p_i$ in $S_{j-1}.$ 
\end{lem}

Also, we have

\begin{lem} Suppose $S$ has an effective, irreducible, and reduced anti-canonical divisor and suppose $(S,\phi)$ is a marked blowup. Let $C$ be a reduced irreducible curve such that $[C] = -\kappa_S$. The base locus $P$ is completely determined by a set of pairs (or base points with heights)  $\{ (q_i, h_i), 1 \le i \le k  \} \subset \pi C \times \mathbb{N}$ where $q_i = q_i^{(1)}$ with height $h_i$.
\end{lem}

\begin{proof}
Let $C$ be the anti-canonical curve.
By the previous Lemma, we have that $q_i^{(j)}$ is the unique intersection of the strict transform of $\pi C$ and the strict transform of the exceptional curve over $q_i^{(j-1)}$ for all $j=2, \dots h_i$ and $i = 1, \dots, k$. 
\end{proof}

Notice that all the base points with heights for a marked blowup belong to some strict transform of $\pi(C)$. Thus we have
 
 \begin{lem}\label{L:selfintersectionreduction} Suppose $S$ has an effective, irreducible, and reduced anti-canonical divisor and suppose $(S,\phi)$ is a marked blowup. Let $C$ be a reduced irreducible curve such that $[C] = -\kappa_S$. If $X\subset S$ is a smooth rational curve of degree $d$ and the intersection $\pi(X) \cap \pi(C)$ consists of $m_i \cdot q_i^{(1)} $ and $n_j \cdot x_j $ (counted with multiplicity) where $x_j \notin P$ , then 
\begin{equation}\label{E:selfintersection}  X \cdot  X = d^2 -  \sum_i \min\{m_i, h_i\}  \ge d( d-3). \end{equation}
\end{lem} 

\begin{proof}
Since $x_j$'s are not in the base locus $P$, $x_j$ would not change the self-intersection number. 
Also, the total number of the intersection of $\pi X$ and $\pi C$ is $3d$ , we have $\sum_i m_i \le 3d$. 
Since $X$ is a smooth curve in $S$, the singularities (if there is any) of $\pi(X)$ are from the transversal intersections.

Now let us focus on $q_i^{(1)} \in P$. 
For simplicity, let us assume that $p_i = q_1^{(i)}$ for $i = 1, \dots h_i$ and thus $p_i$ is the intersection of exceptional curve $\mathcal{F}_i$ over $p_{i-1}$ and the strict transform of $\pi(C)$ in $S_{i-1}$.
Since computation for any base points is identical, we may assume that $m_j = 0$ for all $j \ge 2$. 
Suppose $m_1 =1$. In $S_1$, the strict transform of $\pi X$ intersects $\mathcal{F}_1$ but is disjoint from $(\pi_2 \circ \cdots \circ  \pi_n) C$. Thus the strict transform of $X$ in $S_i$ doesn't contain $p_{i+1}$ for all $i\ge 2$. It follows that  $ X \cdot  X = d^2-1$.
If $m_1=2$, because only singular points $\pi X \in \mathbf{P}^2$  are transversal self-intersections, there are only two possibilities: (1) the strict transform of $\pi X$ in $S_1$ is disjoint from $(\pi_2 \circ \cdots \circ  \pi_n) C$ or (2) the strict transform of $\pi X$ in $S_1$ intersects $(\pi_2 \circ \cdots \circ  \pi_n) C$ at $p_2$. In either case, the strict transform of $\pi X$ in $S_2$ is disjoint from  $(\pi_3 \circ \cdots \circ  \pi_n) C$. 
Continuing this procedure, we see that for $1 \le m_1\le h_1$, we have $ X \cdot  X = d^2-m_1$. If $m_1= h_1+ s$ for some $s>0$, then $ X \cdot  X = d^2-h_1$ since there are no more blowups over an infinitely near point of $p_{h_i}$. 
Thus we have $ X \cdot  X = d^2-\min \{m_1, h_1\}$. Repeating this for $j>1$ if necessary, we have $ X \cdot  X = d^2 - \sum_i \min\{m_i, h_1\} $. Combining the fact that  $\sum_i m_i \le 3d$, we have the inequality in (\ref{E:selfintersection}).
\end{proof}
%
%
%
%
%
%
%

 \subsection{Nodal roots}
 We say $\alpha \in \Phi_n^+$ is a \textit{nodal root} if $\alpha \in \Phi_n^+$ and $\phi(\alpha)$ is effective and irreducible. Let \[ \Phi(\phi)^+ \ := \{ \alpha \in \Phi_n^+; \ \phi(\alpha) \text{ is effective. }\} \] be the set of nodal roots. There are two possibilities. If $\alpha \in \Phi(\phi)^+$ then 
 \begin{itemize}
 \item either $\alpha = \ie_i - \ie_j$ with $i<j$ and $p_j$ is infinitely near to $p_i$ of the first order
 \item or $\phi(\alpha)= [D]$ where $D$ projects to a smooth rational curve in $S$. In this case, we say $\alpha$ is a geometric nodal root. 
 \end{itemize}
 
 \begin{prop}\label{P:nodaldegree} Suppose $S$ has an effective, irreducible, and reduced anti-canonical divisor and suppose $(S,\phi)$ is a marked blowup. Let $C$ be a reduced irreducible curve such that $[C] = -\kappa_S$. If $\alpha$ is a geometric nodal root with $\phi(\alpha)= [D]$, then either $\text{deg } \pi D =1 $ or  $\text{deg }\pi D =2 $
 \end{prop}
 
 \begin{proof} Let $d =\text{deg }\pi D$. By Lemma \ref{L:selfintersectionreduction}, we see that $D\cdot D  \ge d(d-3)$. Since $\alpha$ is a nodal root, $D\cdot D = -2$. Thus only possibilities are $d=1$ or $d=2$. 
 \end{proof}

 \begin{cor}\label{C:finite} Suppose $S$ has an effective, irreducible, and reduced anti-canonical divisor and suppose $(S,\phi)$ is a marked blowup. Then there are, at most, finitely many nodal roots. 
 \end{cor}

\begin{proof}
Since there are only finitely many blowups, the set \[ \mathcal{A}=  \{ e_i - e_j, i<j\} \cup \{ de_0 - \sum m_i e_i : d=1,2 , \sum m_i \le 3 d, m_i \ge 0 \}\] is finite. By Proposition \ref{P:nodaldegree}, we see all nodal roots are in $\mathcal{A}$.
 \end{proof}
 
  By Dolgachev and Ortland \cite[Proposition~4]{Dolgachev-Ortland}, we see that a geometric nodal root $\alpha$ is given by 
 \[ \alpha = e_0 - (e_{i_1} + e_{i_2}+e_{i_3}), \quad \text{or} \quad \alpha = 2 e_0 - (e_{i_1} + \cdots +e_{i_6}) \qquad i_j \ne i_k \text{ if } j \ne k. \]

\begin{lem}\label{C:nodalinter}
If $\alpha\ne \beta$ are two nodal roots, then either $\alpha\cdot \beta = 0$ or $\alpha\cdot \beta = 1$.
\end{lem}

\begin{proof}
The statement is obvious if $\alpha = e_i - e_j$.

Now suppose $\phi(\alpha) = D_\alpha, \phi(\beta)= D_\beta$.
If both $D_\alpha$ and $D_\beta$ are lines, $\alpha\cdot \beta = 1$ or $ 0$ because two points determine a line. 

Suppose $\alpha = e_0 - (e_{i_1} + e_{i_2}+e_{i_3})$ and $\beta= 2 e_0 - (e_{j_1} + \cdots+e_{j_6})$. If $\{ i_1, i_2, i_3\} \cap \{ j_1, \dots, j_6\}=\emptyset$. Since a cubic curve is determined by 9 points, we have that $C= D_\alpha + D_\beta$ is reducible. Thus we have $\alpha\cdot \beta = 1$ or $ 0$. 

Lastly suppose $\alpha =2 e_0 - (e_{i_1} + \cdots +e_{i_6})$ and $\beta= 2e_0 - (e_{j_1} + \cdots+e_{j_6})$. If $|\{ i_1, \dots, i_6\} \cap \{ j_1, \dots, j_6\}|\le 3$ then 
we may assume $\{ i_1, \dots, i_6\} =\{1,\dots, 6\}$ and $ \{ j_1, \dots, j_6\} =\{7,\dots,9, j_1, j_2,j_3\}$ where $j_1, j_2,j_3 \in \{1, \dots, n\}$.
Let us consider two degree 4 curves such that $Q_\alpha= D_\alpha+ Q'_\alpha$ and $Q_\beta= D_\beta+ Q'_\beta$ where $Q'_\alpha$ is a conic passing through $p_4, \dots, p_7$ and $p_{10}$ and $Q'_\beta$ is a conic passing through $p_6, \dots p_{10}$ where $p_i$'s are the base point corresponding $e_i$. By \cite{Traves:2013}, there exist a cubic passing through $p_1, \dots, p_{10}$ if and only if six of $\{p_1, \dots, p_{10}\}$ are in $Q_\alpha \cap Q_\beta$. However $|Q_\alpha \cap Q_\beta|\le 5$. It follows that $\alpha\cdot \beta =0$ because $\alpha \ne \beta$
\end{proof}

We say two nodal roots $\alpha, \beta$ are \textit{linked} if either $\alpha\cdot \beta = 1$ or there exists a finite set of nodal roots $\{ \alpha_i, i=1, \dots k\}$ for some $k \in \mathbb{N}$ such that \[ \alpha\cdot \alpha_1 = \alpha_i, \cdot \alpha_{i+1} = \alpha_k \cdot \beta =1.\]
Let us denote $\mathcal{N}^+(\alpha)$ the set of all nodal roots linked to $\alpha$ and let $\ell^+(\alpha) = | \mathcal{N}^+(\alpha)|$, the number of linked nodal roots linked $\alpha$ including $\alpha$. Since any roots $\alpha$ satisfies $\alpha \cdot \kappa_n =0$, it forms a Hirzebruch-Jung string. In fact it is known \cite[Section~III.2]{Barth:2004} that the intersection form on $\mathcal{N}^+(\alpha)$ is given by Dynkin diagram $A_n, D_n, E_6,E_7$ or $E_8$ with $n\ge 1$.

%
%
%
Also let $\Sigma(\alpha)$ be the span of reflections in $\mathcal{N}^+(\alpha)$
\[ \Sigma(\alpha) = \text{Span} \{ s_\beta: \beta \in \mathcal{N}^+(\alpha)\}. \]

\begin{lem}\label{L:nodal0}
Suppose $\alpha$ is a nodal root such that $\alpha\cdot \beta =0$ for all nodal roots $\beta \ne \alpha$. Let $s= s_\alpha$ be the reflection through $\alpha$. Then 
\[ s(\alpha) = - \alpha, s(\beta) = \beta \text{ for all } \beta \in \Phi^+(\phi) \setminus \{\alpha\} \]
\end{lem}

\begin{proof}
This is trivial since $\alpha\cdot \beta =0$.
\end{proof}

\begin{lem}\label{L:nodal1}
Let $\omega \in W_n$. Suppose $\alpha$ is a nodal root with $\ell^+(\alpha)=2$ satisfying $\omega(\alpha) \in \Phi^-_n$. Then there is $s \in \Sigma(\alpha)$ such that 
\[ \omega\circ s(\alpha) \in  \Phi^+_n, \quad \text{ and}\]
\[ \text{either } \quad \omega\circ s(\beta), \omega(\beta) \in \Phi^+_n \quad \text{  or  }\quad  \omega\circ s(\beta), \omega(\beta) \in \Phi^-_n\quad  \text{ for all } \beta \in \Phi^+(\phi) \setminus \{\alpha\} \]
\end{lem}

\begin{proof}
Let $\mathcal{N}^+(\alpha) = \{ \alpha_1, \alpha_2\}$. We may assume $\alpha= \alpha_1$. Let $s_i $ be the reflection through $\alpha_i$.
Note that $s_1(\alpha_2) = \alpha_1+\alpha_2$. Thus $\alpha_1+ \alpha_2 \in \Phi_n$.
Since all roots not in $\mathcal{N}^+(\alpha)$ are orthogonal to $\alpha_1$, we only need to check for $\alpha_1, \alpha_2$.
\begin{itemize}
\item If $\omega(\alpha_2)  \in  \Phi^+_n$ and $\omega(\alpha_1+\alpha_2)  \in  \Phi^-_n$  then let $s= s_1 \circ s_2$. We have 
\[ s(\alpha_1)= \alpha_2, \qquad  s(\alpha_2) = -(\alpha_1+\alpha_2),\qquad \text{ and }\]
 \[ \omega\circ s(\alpha_1) = \omega( \alpha_2) \in \Phi^+_n, \qquad \omega\circ s(\alpha_2) = - \omega (\alpha_1+\alpha_2) \in \Phi^+_n\]
 \item If $\omega(\alpha_2)  \in  \Phi^+_n$ and $\omega(\alpha_1+\alpha_2)  \in  \Phi^+_n$  then let $s= s_1$. We have 
\[ s(\alpha_1)= -\alpha_1, \qquad  s(\alpha_2) = \alpha_1+\alpha_2,\qquad \text{ and }\]
 \[ \omega\circ s(\alpha_1) = -\omega( \alpha_1) \in \Phi^+_n, \qquad \omega\circ s(\alpha_2) =  \omega (\alpha_1+\alpha_2) \in \Phi^+_n\]
\item If $\omega(\alpha_2)  \in  \Phi^-_n$  then let $s= s_1$. We have 
\[ s(\alpha_1)= -\alpha_1, \qquad  s(\alpha_2) = (\alpha_1+\alpha_2),\qquad \text{ and }\]
 \[ \omega\circ s(\alpha_1) =- \omega( \alpha_1) \in \Phi^+_n, \qquad \omega\circ s(\alpha_2) = \omega (\alpha_1+\alpha_2) = \omega(\alpha_1) + \omega(\alpha_2) \in \Phi^-_n\]
\end{itemize}
We covered all possibilities.
\end{proof}

Notice that when $\ell^+(s)$ increases from $1$ to $2$,  $s$ has to changed by multiplying $s_{2}$ on the right if two conditions ($ \omega(\alpha_2)\in  \Phi^+_n$ and $\omega (\alpha_1+ \alpha_2)  \in \Phi^-_n$) holds. 

\begin{lem}\label{L:nodal2}
Let $\omega \in W_n$. Suppose $\alpha$ is a nodal root with $\ell^+(\alpha)=3$ satisfying $\omega(\alpha) \in \Phi^-_n$. Then there is $s \in \Sigma(\alpha)$ such that 
\[ \omega\circ s(\alpha) \in  \Phi^+_n, \quad \text{ and}\]
\[ \text{either } \quad \omega\circ s(\beta), \omega(\beta) \in \Phi^+_n \quad \text{  or  }\quad  \omega\circ s(\beta), \omega(\beta) \in \Phi^-_n\quad  \text{ for all } \beta \in \Phi^+(\phi) \setminus \{\alpha\} \]
\end{lem}

\begin{proof}
Let $\mathcal{N}^+(\alpha) = \{ \alpha_j, j=1,2,3\}$. 
Notice that$s_1(\alpha_2) = \alpha_1+ \alpha_2$ and $s_2(s_1(\alpha_3)) = \alpha_1+\alpha_2+\alpha_3$. Thus, using other indexes, we see that sums of distinct $\alpha_i$'s are in $\Phi_n$. We have two cases:

\vspace{1ex}
\textbf{Case 1:  $\alpha= \alpha_1$} There are three possibilities:
\begin{itemize}
\item If $\omega(\alpha_2) \in \Phi^+_n$,$\omega(\alpha_3) \in \Phi^+_n$, $\omega(\alpha_1+\alpha_2) \in \Phi^-_n$, and $\omega(\alpha_1+\alpha_2+\alpha_3) \in \Phi^-_n$, then let $s=s_1\circ s_2\circ s_3$. We have 
\[ s(\alpha_1) = \alpha_2, s(\alpha_2) = \alpha_3, s(\alpha_3) = - (\alpha_1+ \alpha_2+\alpha_3), \qquad \text{ and} \]
 \[ \omega\circ s(\alpha_1) = \omega( \alpha_2) \in \Phi^+_n, \  \omega\circ s(\alpha_2) = \omega( \alpha_3) \in \Phi^+_n, \  \omega\circ s(\alpha_3) = - \omega (\alpha_1+\alpha_2+\alpha_3) \in \Phi^+_n.\]
 Since it is not hard to check, we will give only the formula for $s$ for the remaining cases. Also, for all cases, the possibilities are listed such that each condition assumes that the conditions above possibilities have failed. 
 \item If $\omega(\alpha_2) \in \Phi^+_n$,$\omega(\alpha_3) \in \Phi^-_n$, and $\omega(\alpha_1+\alpha_2) \in \Phi^-_n$, then $ s= s_1 \circ s_2$
 \item All other cases: We have $s= s_1$.
\end{itemize}

\vspace{1ex}
\textbf{Case 2:  $\alpha= \alpha_2$} There are five possibilities:
\begin{itemize}
\item If $\omega(\alpha_1), \omega(\alpha_3) \in \pos$, and $ \omega(\alpha_1+\alpha_2+\alpha_3) \in \nee$, then \[s= s_2 \circ s_1\circ s_3 \circ s_2\]
\item If $\omega(\alpha_1), \omega(\alpha_3) \in \pos$, $\omega(\alpha_1+\alpha_2), \omega(\alpha_2+\alpha_3)\in \nee$, and $\omega(\alpha_1+\alpha_2+\alpha_3) \in \pos$, then \[s= s_2 \circ s_3\circ s_1\]
\item If $\omega(\alpha_3)\in \pos$, $\omega(\alpha_2+\alpha_3) \in \nee$ but not in above two cases, then \[s= s_2\circ s_3\]
\item If $\omega(\alpha_1) \in \pos$, $\omega(\alpha_1+\alpha_2) \in \nee$ but not in above three cases, then \[s= s_2 \circ s_1\]
\item All other cases, we have $s=s_2$.
\end{itemize}
\end{proof}

\begin{lem}\label{L:nodal4}
Let $\omega \in W_n$. Suppose $\alpha$ is a nodal root with $\ell^+(\alpha)=4$ such that the intersection form of $\mathcal{N}^+(\alpha)$ is given by $D_4$. Suppose $\omega(\alpha) \in \Phi^-_n$. Then there is $s \in \Sigma(\alpha)$ such that 
\[ \omega\circ s(\alpha) \in  \Phi^+_n, \quad \text{ and}\]
\[ \text{either } \quad \omega\circ s(\beta), \omega(\beta) \in \Phi^+_n \quad \text{  or  }\quad  \omega\circ s(\beta), \omega(\beta) \in \Phi^-_n\quad  \text{ for all } \beta \in \Phi^+(\phi) \setminus \{\alpha\} \]
\end{lem}

\begin{proof} Let $\mathcal{N}^+(\alpha) = \{ \alpha_i, i=1,2,3,4\}$ and $\alpha_3 \cdot \alpha_i = 1$ for all $i \ne 3$. Due to the symmetry, we have two cases:

\vspace{1ex}
\textbf{Case 1:  $\alpha= \alpha_1$} There are six possibilities:
\begin{itemize}
\item If $\omega(\alpha_i) \in \Phi^+_n$ for $i \ne 1$ and $\omega (\alpha_1 + \alpha_2 + \alpha_3 +\alpha_4) \in \Phi^-_n$, then 
\[ s= s1 \circ s_3 \circ s_2 \circ s_4 \circ s_1 \circ s_3\]
 \item If $\omega(\alpha_i) \in \Phi^+_n$ for $i \ne 1$, and $\omega(\alpha_1+\alpha_2+\alpha_3), \omega(\alpha_1+\alpha_3+\alpha_4) \in \Phi^-_n$, then \[ s= s_1 \circ s_3\circ s_2\circ s_4\]
 \item If $\omega(\alpha_i) \in \Phi^+_n$ for $i =3,4$, and $\omega(\alpha_1+\alpha_3+\alpha_4) \in \Phi^-_n$, then \[ s= s_1 \circ s_3\circ s_4\]
  \item If $\omega(\alpha_i) \in \Phi^+_n$ for $i =2,4$, and $\omega(\alpha_1+\alpha_2+\alpha_4) \in \Phi^-_n$, then \[ s= s_1 \circ s_4\circ s_2\]
   \item If $\omega(\alpha_3) \in \Phi^+_n$, and $\omega(\alpha_1+\alpha_3) \in \Phi^-_n$, then \[ s= s_1 \circ s_3\]
 \item All other cases: We have $s= s_1$.
\end{itemize}

\vspace{1ex}
\textbf{Case 2: $\alpha=\alpha_3$} There are twelve possibilities:
\begin{itemize}
\item If $\omega(\alpha_i) \in \Phi^+_n$ for $i \ne 3$ and $\omega (\alpha_1 + \alpha_2 +2  \alpha_3 +\alpha_4) \in \Phi^-_n$, then 
\[ s= s_3 \circ s_2 \circ s_1 \circ s_4 \circ s_3 \]
 \item If $\omega(\alpha_i) \in \Phi^+_n$ and $\omega(\alpha_i+\alpha_3)  \in \Phi^-_n$ for $i \ne 3$, then \[ s= s_1 \circ s_3\circ s_2\circ s_1\circ s_4\]
 \item If $\omega(\alpha_i) \in \Phi^+_n$ for $i =1,2$, and $\omega(\alpha_1+\alpha_2+\alpha_3) \in \Phi^-_n$, then \[ s= s_3 \circ s_2\circ s_1\circ s_3\]
  \item If $\omega(\alpha_i) \in \Phi^+_n$ for $i =1,4$, and $\omega(\alpha_1+\alpha_4+\alpha_3) \in \Phi^-_n$, then \[ s= s_3 \circ s_1\circ s_4\circ s_3\]
   \item If $\omega(\alpha_i) \in \Phi^+_n$ for $i =2,4$, and $\omega(\alpha_2+\alpha_4+\alpha_3) \in \Phi^-_n$, then \[ s= s_3 \circ s_2\circ s_4\circ s_3\]
  \item If $\omega(\alpha_i) \in \Phi^+_n$ for $i =1,2$, and $\omega(\alpha_1+\alpha_3),\omega(\alpha_2+\alpha_3) \in \Phi^-_n$, then \[ s= s_3 \circ s_2\circ s_1\]
    \item If $\omega(\alpha_i) \in \Phi^+_n$ for $i =1,4$, and $\omega(\alpha_1+\alpha_3),\omega(\alpha_4+\alpha_3) \in \Phi^-_n$, then \[ s= s_3 \circ s_4\circ s_1\]
      \item If $\omega(\alpha_i) \in \Phi^+_n$ for $i =2,4$, and $\omega(\alpha_2+\alpha_3),\omega(\alpha_4+\alpha_3) \in \Phi^-_n$, then \[ s= s_3 \circ s_2\circ s_4\]
   \item If $\omega(\alpha_1) \in \Phi^+_n$, and $\omega(\alpha_1+\alpha_3) \in \Phi^-_n$, then \[ s= s_3 \circ s_1\]
      \item If $\omega(\alpha_2) \in \Phi^+_n$, and $\omega(\alpha_2+\alpha_3) \in \Phi^-_n$, then \[ s= s_3 \circ s_2\]
   \item If $\omega(\alpha_4) \in \Phi^+_n$, and $\omega(\alpha_4+\alpha_3) \in \Phi^-_n$, then \[ s= s_3 \circ s_4\]
 \item All other cases: We have $s= s_1$.
\end{itemize}

\end{proof}

\begin{lem}\label{L:nodaln}
Let $\omega \in W_n$. Suppose $\alpha$ is a nodal root with $\ell^+(\alpha)=k\ge 4$ satisfying $\omega(\alpha) \in \Phi^-_n$. Then there is $s \in \Sigma(\alpha)$ such that 
\[ \omega\circ s(\alpha) \in  \Phi^+_n, \quad \text{ and}\]
\[ \text{either } \quad \omega\circ s(\beta), \omega(\beta) \in \Phi^+_n \quad \text{  or  }\quad  \omega\circ s(\beta), \omega(\beta) \in \Phi^-_n\quad  \text{ for all } \beta \in \Phi^+(\phi) \setminus \{\alpha\} \]
\end{lem}

\begin{proof}
Let $\mathcal{N}^+(\alpha) = \{ \alpha_j, j=1, \dots, k\}$. 
In this case, the intersection form of $\mathcal{N}^+(\alpha)$ is given by the Dynkin diagram extending one arm of $A_3$ in Lemma \ref{L:nodal2} or $D_4$ in Lemma \ref{L:nodal4}. We use induction on $\ell^+(\alpha)$. Suppose the statement is true for $\ell^+(\alpha)=k-1$. Assume $\ell^+(\alpha)=k$. Since $\alpha_i \cdot \alpha_{i +\pm 2} =0$, we use the same $s$ in $\ell^+(\alpha)=k$ by changing indexes if necessary.
If there are three arms in the Dynkin diagram, then one of the arms has to have a length of $2$. Because of symmetry, it is sufficient to consider the case $\alpha= \alpha_1$ or $\alpha= \alpha_2$. 
Let $s'$ be the composition of reflections in case $\mathcal{N}^+(\alpha) = \{ \alpha_j, j=1, \dots, k-1\}$.

\vspace{1ex}
\textbf{Case 1:  $\alpha= \alpha_1$}
\begin{itemize}
\item If $s_{k-1}$ is in $s'$ and if $\omega(\alpha_k) \in \pos$ and $\omega(\sum_{i=1}^k \alpha_i) \in \nee$, then \[s= s' \circ s_k\]
\item Otherwise $ s= s'$
\end{itemize}

\vspace{1ex}
\textbf{Case 2:  $\alpha= \alpha_2$}
\begin{itemize}
\item If $s_{k-1}$ is in $s'$, $\omega(\alpha_j) \in \pos$ for all $j \ne 2$, and $\omega(\sum_{i=1}^k \alpha_i) \in \nee$,   then \[s= s' \circ s_k\circ s_{k-1}\]
\item If $s_{k-1}$ is in $s'$ and if $\omega(\alpha_k) \in \pos$ and $\omega(\sum_{i=2}^j \alpha_i) \in \nee$ but not in above, then \[s= s'\circ s_k\]
\item Otherwise $ s= s'$
\end{itemize}
\end{proof}

For a maked blowup $(S, \phi)$, let $\Sigma_\phi$ be the span of all reflections through nodal roots: \[\Sigma_\phi = \text{Span} \{ s_\alpha: \alpha \in \Phi^+(\phi)\}\]

 \begin{prop}\label{P:preservingpositivity} Suppose $S$ has an effective, irreducible, and reduced anti-canonical divisor and suppose $(S,\phi)$ is a marked blowup.
 Let $C$ be a reduced irreducible curve such that $[C] = -\kappa_S$ and $\pi(C)$ is a cubic curve with one singularity.
 Let $\omega \in W_n$. Then there is $s \in \Sigma_\phi$ such that $\omega\circ s(\alpha) \in \pos$ for all nodal roots $\alpha$. 
 \end{prop}
 
 \begin{proof}
 By Lemmas \ref{L:nodal0} -- \ref{L:nodaln}, there is $s\in \Sigma_\phi$ such that $\omega \circ s$ changes the positivity of exactly one nodal root. Since there are only finitely many nodal roots, by changing one at a time, we can map all nodal roots to positive roots.
  \end{proof}

\subsection{Isomorphisms}
 Let $(S, \phi)$ and $(S', \phi')$ be two marked blowups. If there is a biholomorphism $F:S \to S'$ such that the induced map $F_*: Pic(S) \to Pic(S')$ satisfies $F_* \circ \phi = \phi'$, then we say $(S, \phi)$ and $(S', \phi')$  are \textit{isomorphic}, $(S, \phi)\cong (S', \phi')$.  It is known \cite{Nagata2,Dolgachev-Ortland} that if $(S,\phi)$ and $(S,\phi\circ \omega)$ are marked blowups, then $\omega \in W_n$.
 
 Let 
 \[ W(S,\phi) = \{ \omega\in W_n: (S, \phi\circ \omega) \text{ is a marked blowup. }\}\]
 McMullen \cite{McMullen:2007} showed that $W(S,\phi) = W_n$ provided that there is no nodal root.
 \begin{thm} \cite[Theorem~5.4]{McMullen:2007}
 If $(S,\phi)$ has no nodal roots, then $W(S, \phi)= W_n$.
 \end{thm}
 
 \begin{proof}
 If $(S,\phi)$ has no nodal roots, then the set $P$ is a set of distinct points; in other words, there are no non-geometric nodal roots. So we can directly apply McMullen's result. 
 \end{proof}
 
 In case there are nodal roots, Harbourne \cite{Harbourne:1985} gives an equivalent condition for $\omega \in W_n$ to be an element in $W(S, \phi)$.
 
 \begin{thm}\cite[Proposition~2.4]{Harbourne:1985}\label{T:harbourne}
 Let $\omega \in W_n$. Then $\omega \in W(S, \phi)$ if and only if $\omega^{-1} \Phi(\phi)^+ \subset \Phi_n^+$.    
 \end{thm}
 
 \begin{proof}
 If there is a sequence of blowups $S$: 
  \[ \pi : S = S_n \xrightarrow{\pi_n} S_{n-1} \xrightarrow{\pi_{n-1}} \ \  \cdots  \ \  \xrightarrow{\pi_3}S_2 \xrightarrow{\pi_2} S_1 \xrightarrow{\pi_1} S_0 = \mathbf{P}^2 \] 
 where $\pi_i: S_i \to S_{i-1}$ is a blowup of a point $p_i \in S_{i-1}$, then
  the collection of $\{ \ve_0, \dots, \ve_n\} \subset Pic(S)$ is called an \textit{exceptional configuration} in \cite{Harbourne:1985}. Thus, if a collection $\mathcal{E}=\{ \ve_0, \dots, \ve_n\} $ is an exceptional configuration, there is a marked blowup $(S,\phi)$ such that $\ve_i = \phi(e_i)$ for $i=0, \dots, n$. Also, each marked blowup has its exceptional configuration. The assertion is stated and proved using exceptional configurations in \cite{Harbourne:1985}
  \end{proof}
  
Suppose $S$ has an effective, irreducible, and reduced anti-canonical divisor, and suppose $(S,\phi)$ is a marked blowup. Let $\omega \in W_n$ such that there exists a vector $v \in \mathbb{Z}^{1,n}$ and a real number $\lambda>1$ such that $\omega v = \lambda v$. Then it is known \cite{Harbourne:1985} that $W(S, \phi)$ is infinite. The following theorem is written using the notations in this article. 

\begin{thm}\cite[Theorem~3.1]{Harbourne:1985}
Suppose $(S, \phi)$ is a marked blowup and that $\text{rank} Pic (S) \ge 10$. Let $d$ be the rank of the submodule of $Pic S$ spanned by the nodal roots. Then $W(S, \phi)$ is finite if and only if $d=-1+\text{rank\,}Pic S$ and $S$ has only finitely many nodal roots. 
\end{thm}

%
  
 \subsection{Automorphisms}

  Suppose a rational surface $S$ admits two marked blowups, $(S, \phi)$ and $(S, \phi')$. Also suppose $(S,\phi) \cong (S, \phi')$. Then $\phi' = \phi\circ \omega$ for some $\omega \in W_n$ and there exists an automorphism on $S$. Let 
  \[ \text{Aut}(S,\phi) = \{ \omega \in W(S, \phi) :(S,\phi) \cong (S, \phi \circ \omega) \} \]

%
%

\section{Marked cuspidal cubic}\label{S:cubic}
Suppose the anti-canonical class of $\pi:S \to \p^2$ is effective, irreducible, and reduced. Then there is an irreducible curve $C$ such that $[C] = - \kappa_S$ and the every center of blowup, $p_i \in S_{i-1}$ is in the projection of $C$, $\pi_i \circ \cdots \circ \pi_n \, (C) \subset S_{i-1}$. Thus, a marked blowup $(S, \phi)$ is determined by a set of points $\{p_1, \dots, p_n\}$, the markings on  $C$ together with their heights. For an irreducible and reduced cubic curve $X \subset \p^2$, let us denote $X^*$, the set of smooth points of $X$. 

\subsection{Marked Cubic curves.}\label{SS:gmc} Similar to generalized Marked Blowups, we make a slight modification in the definition of Marked cubic curves in \cite{McMullen:2007} to allow iterated blowups.  

\begin{defn}\label{D:gmc}
A marked cubic curve $(X, \rho)$ is an abstract curve $X$ equipped with a homomorphism $\rho: \zm \to Pic(X)$ such that 
\begin{enumerate}
\item $\rho(e_0)$ provides an embedding $X \hookrightarrow \p^2$, making $X$ into a cubic curve, 
\item there are a positive integer $ 3 \le k\le n$ and a set of pairs \[ P= \{(p_i, h_i) : p_i \ne p_j \text{ for }  1 \le i \ne j \le k \} \subset X^* \times \mathbb{N}_{\ge 1} \] such that
\begin{itemize}
\item the sum $\sum_{i=1}^k h_i= n$,
\item there are $k$ pairwise disjoint subsets $I_i$ of $\{1, \dots, n\}$ such that $|I_i| = h_i$, and 
\item $\rho(e_j) = [p_i]$ for $j \in I_i$. 
\end{itemize}
\end{enumerate}
\end{defn}

We call a positive integer $h(p_i):=h_i$ assigned to $p_i$  above a \textit{height} of $p_i$. As defined in \cite{McMullen:2007}, we say two marked cubics $(X,\rho)$ and $(X',\rho')$ are isomorphic, $(X, \rho) \cong (X',\rho')$ if there exists a biholomorphic map $f:X \to X'$ such that $\rho' = f_* \circ \rho$. Notice that a marked cubic in \cite{McMullen:2007} only allows $h_i = 1$ for all $i=1, \dots, n$. Since the heights of base points could be bigger than $1$, we have the following
\begin{lem}\label{L:cubicheight}
If two marked cubics $(X,\rho)$ and $(X',\rho')$ are isomorphic, then there is $1-1$ correspondence between the sets of base points with their heights $P = \{ (p_i, h_i), 1 \le i \le k \}$ and $P'= \{ (p_i', h_i), 1 \le i \le k \}$ satisfying $p_i \ne p_j$ and $p'_i \ne p'_j$ for all $i \ne j$. 
\end{lem}
\begin{proof}
If $p_i$ has a height $h_i$, then there is a subset $I_i \subset \{ 1, \dots, n \}$ such that $\rho(e_j) = [p_i]$ for all $j \in I_i$. It follows that $\rho'(e_j) = \rho'(e_s)$ for all $i, s \in I_i$. 
\end{proof}
We also define \[ W(X,\rho) = \{ \omega \in W_n: (X, \rho\circ \omega) \text{ is a marked cubic. }\}\]
and \[ \text{Aut}(X, \rho) = \{ \omega \in W(X, \rho) : (X, \rho) \cong (X, \rho\circ \omega)\}.\]

\subsection{Anti-Canonical surfaces} The main object of this article is a rational surface with an effective, irreducible, and reduced anti-canonical divisor. Let us denote $(S, Y, \phi)$ a marked blowup with a unique anti-canonical curve $Y$. If  $(S, \phi)\cong (S', \phi')$, then there is a biholomorphism $F:S \to S'$ such that $F_*$ maps $\phi(e_i)$ to $\phi'(e_i)$. 
\begin{defn}
We say $(S,Y, \phi)$ and $(S',Y' \phi')$ are isomorphic, $(S,Y, \phi)\cong (S',Y' \phi')$ if there is a biholomirphic map $F:S \to S'$ such that $F_* \circ \phi = \phi'$ and $F(Y) = Y'$. 
\end{defn}

And we define 
\[ \text{Aut}(S,Y, \phi) = \{ \omega \in W_n : (S, Y, \phi) \cong (S, Y, \phi \circ \omega) \}.\]
If $\omega \in \text{Aut}(S,Y, \phi)$, then there is an automorphism $F:S \to S$ preserving $Y$ such that $\phi\circ \omega = F_* \circ \phi$ on $\mathbb{Z}^{1,n}$. In this case, we say $\omega$ is \textit{realized} by $F \in \text{Aut}(S,Y)$.

\subsection{Cuspidal Cubic}
Suppose that $X$ is a cuspidal cubic with a set $X^*$ of smooth points on $X$. 
Then $Pic_0(X) = \mathbb{C}$, an additive group. If $f \in Aut(X)$ then $f(t) = a t+b$.  We call $a$ the \textit{determimant} of $f$ as in \cite{McMullen:2007}. Note that if $\omega$ is a meromorphic two form on $X$, then $f^* \omega = a \omega$. 

Suppose $(X, \rho)$ is a marked cuspidal cubic. We have $\deg(\rho(u)) = - u \cdot \kappa_n$ for $u \in \mathbb{Z}^{1,n}$ and let $\rho_0$ be the restriction \[ \rho_0 : Ker(\text{deg} \circ \rho) \to Pic_0(X). \]

\subsubsection{Blowup of a Marked cuspidal cubic}
If $(X, \rho)$ is a marked cubic with markings $P= \{ (p_i, h_i) \}$ and the embedding $X \subset \p2$ determined by $\rho(e_0)$ is a cuspidal cubic Let $(S, \phi)$ be the marked blowup with base points \[P =\cup_i   \{ p_i^{(j)}, j=1, \dots h_i\}\]  where \[ p_i^{(1)} = p_i \in X \ \ \text{  and  }\ \  p_i^{(j+1)} = \mathcal{F}(p_i^{(j)}) \cap \tilde X, \ j=2, \dots, h_i-1\] where $\mathcal{F}(p_i^{(j)})$ is the exceptional curve over $p_i^{(j)}$ and $\tilde X$ is the strict transform of $X$. Let $Y$ be the strict transform of $X$ in $S$. Then $[Y] = - \kappa_S$. Thus we have \[ (S, Y, \phi) \ =\ \text{B}\ell (X, \rho)\] and we call this the blowup of $(X, \rho)$ (or a marked pair) as in \cite{McMullen:2007}.

On the other hand, suppose $(S, Y, \phi)$ is a marked blowup with the unique anti-canonical curve $Y$ such that $X=\pi(Y)$ is a cuspidal cubic. Suppose the base locus is given by \[P= \ \{ p_1, \dots p_n\} \ \ =\  \{ q^{(j)}_i, 1 \le j \le h_i, 1 \le i \le k\}.\]   Then with the restriction map $r: Pic(S) \to Pic(Y)$, we have a marking $\rho : \mathbb{Z}^{1,n} \to Pic(X)$ such that $\rho(e_m) = [ \pi (p_m)] $ where $ \pi (p_m) = q_i$ if $p_m = q_i^{(j)}$ for some $j =1, \dots, h_i$. Thus $(S,Y, \phi)$ determines a marked cubic $(X, \rho)$. 

\subsubsection{Nodal roots} Suppose $(S,Y, \phi) = \text{B}\ell (X, \rho)$. Suppose $\alpha$ is a nodal root, then we have either $\alpha =e_i-e_j$ with $i<j$ or $\alpha$ is a geometric nodal root. For the first case, we have $\deg\circ \rho(\alpha) = -\alpha \cdot \kappa_n = 0$ and for the second case we also have  $\deg\circ \rho(\alpha)=0$ by \cite[Theorem~6.6]{McMullen:2007}. Thus we have

\begin{prop}\label{P:rhozeronodal} If $\alpha$ is a nodal root, then $\rho_0(\alpha) = 0$. 
\end{prop}

\section{Automorphisms}\label{S:auto}
Let $(S, Y, \phi) = \text{B}\ell(X,\rho)$ be a marked blowup with the unique anti-canonical curve $Y$ such that  $X= \pi Y$ is a cuspidal cubic. Let $\text{Aut}(S, Y)$ be the group of automorphism of $S$ fixing $Y$.  We have 
\begin{thm}\cite[Theorem~6.1]{McMullen:2007}\label{T:autgroup}
Suppose $(S, Y, \phi) = \text{B}\ell(X,\rho)$. Then we have \[ \text{Aut}(S, Y, \phi)\ =\ \text{Aut}(X, \rho) \cap W(S, \phi) \]
\end{thm}

\begin{proof}
Since $Y$ is the unique anti-canonical curve and $X= \pi(Y)$ is irreducible, the proof is essentially the same as in \cite{McMullen:2007}. Suppose $\omega \in \ \text{Aut}(X, \rho) \cap W(S, \phi)$. Then $\phi$ and $\phi' = \phi\circ \omega$ are two markings on $S$. Let $\pi'$ be the projection corresponding to $\phi'$. We can choose the embedding of $\pi'(Y) \subset \p^2$ so that $(S, Y, \phi') = \text{B}\ell(X, \rho'=\rho\circ \omega)$.
Since $\omega \in \text{Aut}(X, \rho)$, there is an automorphism $g$ on $X$ such that $g(p_i) = p'_j$ with $h(p_i) = h(p'_j)$. Thus $g$ lifts to an automorphism $F$ on $S$ such that $\pi \circ g = F \circ \pi'$ and $F(Y) = Y$. Thus $\omega \in  \text{Aut}(S, Y, \phi)$. 
If $\omega \in  \text{Aut}(S, Y, \phi)$, then by definition $\omega \in W(S, \phi)$ and there is an automorphism $F$ preserving $Y$.  It follows that there is a birational map $f$ on $\p^2$ such that the restriction $f|X$ is an automorphism on $X$.
\end{proof}

For any automorphism $F$ on $S$ preserving $Y$, we have a constant $\delta(F) = det DF_p$ for all fixed points $p \in S\setminus Y$. We call this constant $\delta(F)$ the determinant of $F$. As a special case of Theorem $6.5$ in \cite{McMullen:2007}, we have 
\begin{thm}\label{T:determinant} For any $F \in \text{Aut}(S,Y)$, we have the determinant of $f|X = \delta(F)$, that is \[ f|X\,(t) = \delta(F) t + b,\quad \quad  t \in \mathbb{C}\]
\end{thm}

\section{Realizing Weyl Spectrum}\label{S:realization}
Let $\mathbb{C}^{1,n} = \mathbb{Z}^{1,n} \otimes \mathbb{C}$ with the complex Minkowski form. For each $v \in \mathbb{C}^{1,n}$, we let \[ W_n^v = \{ \omega \in W_n: [v] \in \mathbb{C}^{1,n}/\mathbb{C} \kappa_n \ \text{ is an eigenvector for } \omega \} \]

\subsection{Blowup of a marked cuspidal cubic}The following construction is essentially the same as the one in \cite[Section~7]{McMullen:2007}. One difference is that we allow the iterated blowups over a point in $\p^2$. Let $X= \{ p(t)=[1:t: t^3] , t \in \mathbb{C} \cup \{\infty\} \}$ be the cuspidal cubic with the smooth locus $X^*$. For each $v \in \mathbb{C}^{1,n}$, define 
\[  3 t_0 = v\cdot e_0, \qquad \  t_i = v \cdot e_i,   \quad  \text{and} \quad  p_i = p(t_i-t_0) \quad \text{ for } i >0.  \]
Let $\{q_i : 1 \le i \le k\}$ be the set of distinct points in $\{p_i: 1 \le i \le n\}$, and $h_i:=h(q_i) = \# \{ j : p_j = q_i\}$. Then the set of pairs $\{ (q_i, h_i), 1 \le i \le k \}$ determines a marking $\rho^v: \mathbb{Z}^{1,n} \to Pic(X)$, and \[ \rho^v_0(u) = (u \cdot v) [ p(1)-p(0)]. \] Since $Pic_0(X) \cong \mathbb{C}$ and the height of each base point is given by the occurrence of coordinates of $v$, from Lemma \ref{L:cubicheight} we have \[ (X, \rho^v) \cong (X, \rho^{u})\ \ \Leftrightarrow u = av + \mu \kappa_n,  \text{for some } a, \mu \in \mathbb{C}. \] Thus, we have 
\begin{prop}\label{P:autcubic}
Let $v \in \mathbb{C}^{1,n}$. Then we have \[ \text{Aut}(X, \rho^v) = W_n^v. \]
\end{prop}

 Let $(S^v, Y^v, \phi^v) = \text{B}\ell(X, \rho^v)$ where $Y^v$ is the strict transform of $X$.

\subsection{$\text{Aut}(S^v,Y^v,\phi^v)$} For a given $\omega \in W^{ess}_n$, let $\lambda(\omega)$, the spectral raidius. It is known \cite{Bedford-Kim:2006,gross2009cyclotomic} that the spectral radius of the coxeter element of $W_n$ is the largest real root of \[ \chi_n(t) = t^n(t^3-t-1) + (t^3+t^2-1) .\] Thus if $\omega \in W^{ess}_n$ for $n\ge 10$, then $\lambda(\omega)>1$. 
By Nagata \cite{Nagata, Nagata2}, it is well known that if $\lambda(\omega)>1$, then $\lambda(\omega)$ is a Salem number, and eigenvalues of $\omega$ are either Galois conjugates of $\lambda(\omega)$ or a root of unity. A \textit{Salem number} is an algebraic integer $>1$ whose minimal polynomial is reciprocal of degree at least $4$ and has exactly two roots outside the unit circle. 
In fact, we have,
\begin{thm} \cite{Nagata, Nagata2}
Suppose $\omega \in W_n$ and let $\lambda_\omega$ denote the spectral radius of $\omega$. If $\lambda_\omega>1$, then the characteristic polynomial $\chi_\omega(t)$ is given by a product of a Salem polynomial $Q(t)$  and a product $R(t)$ of cyclotomic factors \[ \chi_\omega(t) \ =\ Q(t) R(t). \]
\end{thm}
Since $\lambda(\omega)>1$ is a Salem number, we have 
\begin{lem}\label{L:nobasepoints}
Suppose $\omega \in W_n^{ess}$ for some $n\ge 10$. Let $v$ be the eigenvector corresponding to the spectral radius $\lambda(\omega)$. Then, the marked cuspidal cubic $(X, \rho^v)$ has at least three base points. 
\end{lem}

\begin{proof}Recall that the set of base points of $(X, \rho^v)$ is given by the set of distinct points in $p(v \cdot e_i - \frac{1}{3} v \cdot e_0)$. Thus, the number of distinct $v \cdot e_i$ is the number of base points. 
If there is a unique base point, then $v = \{ a, b, b, \dots, b\}$ for some $a,b \in \mathbb{C}$. Thus there are $n-1$ linearly independent nodal roots $e_i-e_{i+1}, i=1, \dots, n-1$. Similarly, if there are two distinct base points, there are $n-2$ linearly independent nodal roots. $\omega \in W_n^{ess}$ with $n\ge 10$, $\lambda(\omega)>1$. Since every zero of a Salem polynomial is simple, we see that the subspace spanned by eigenvectors corresponding $\lambda(\omega)$ and $1/\lambda(\omega)$ has dimension $2$ in $V$. Thus, there can not be more than $n-3$ linearly independent nodal roots. 
\end{proof}

  Including the trivial example, we have a rational surface automorphism with $\delta(f)=1$. Hence, let us focus on the existence of a realizable element $\omega$ with the spectral radius $\lambda>1$, that is, $\omega \in W_n^{ess}, n\ge 10$. 
Because of Theorem \ref{T:autgroup} and Proposition \ref{P:autcubic},  we have $\text{Aut}(S^v,Y^v,\phi^v) = W_n^v \cap W(S^v, \phi^v)$. 
To determine $W(S, \phi)$, let us consider two cases. 

\vspace{1ex}
\textbf{Case 1. $0 \notin v \cdot \Phi_n$}
Let $v = [v_0, v_1, \dots, v_n] \in \mathbb{C}^{1,n}$. If $0 \notin v \cdot \Phi_n$, then $v_i \ne v_j$ for all $1 \le i\ne  j \le n$. It follows that if $\alpha$ is a nodal root in $(S^v,Y^v,\phi^v)$, then $\alpha$ is a geometric nodal roots. Thus the same argument as in Theorem 6.6 in \cite{McMullen:2007} gives us the following:

\begin{thm}\cite[Theorem~6.6]{McMullen:2007}\label{T:zeronodal}
If $\alpha$ is a geometric nodal root, then $\rho^v_0(\alpha) =0$.
\end{thm}

\begin{proof}
Suppose $\alpha$ is a nodal root in $(S^v,Y^v,\phi^v)$, then $\alpha$ is a geometric nodal roots . Thus there is a smooth rational curve $D$ such that $[\pi(\alpha) ]= [D]$. Since $X$ is a cuspidal cubic, the strict transform $Y^v$ is a singular rational curve. Thus $D \not\subset Y$ and $\rho^v(\alpha) = [D \cap Y] $ is degree zero, i.e. $\rho^v_0(\alpha) = 0$.
\end{proof}

\begin{thm}\label{T:nonodal}
If $0 \notin v \cdot \Phi_n$, then $W(S^v, \phi^v)= W_n$
\end{thm}

\begin{proof}
Since $0 \not\in v \cdot \Phi_n$, we have $ 0 \ne \rho^v_0(u) = (u \cdot v) [ p(1)-p(0)]$ for all $u \in L_n$. Suppose $\alpha$ is a nodal root in $(S^v,Y^v,\phi^v)$, then $\alpha$ is a geometric nodal roots . Thus, by Theorem \ref{T:zeronodal}, we see that there are no nodal roots.
Since there is no nodal roots, by Theorem \ref{T:harbourne}, we see that $W_n = W(S^v, \phi^v)$.
\end{proof}

\vspace{1ex}
\textbf{Case 2. $0 \in v \cdot \Phi_n$}
Again let $v = [v_0, v_1, \dots, v_n] \in \mathbb{C}^{1,n}$. 
If there are no nodal roots, then again, by Theorem \ref{T:harbourne}, we see that $W_n = W(S^v, \phi^v)$. If $\alpha$ is a nodal root, there are two cases. If  $\alpha= e_i - e_j$ then we have $\phi^v(e_j) \rho^v(e_j)= \phi^v(e_i)= \rho^v(e_i)$ and thus $v_i = v_j$. It follows that $\alpha \cdot v = 0$. If $\alpha$ is a geometric nodal root, then by Theorem \ref{T:zeronodal}, we have $\alpha \cdot v =0$. Thus we have
\begin{lem}\label{L:nodalv}
If $\alpha$ is a nodal root, then  $\alpha \cdot v =0$.
\end{lem}

Since the curve $Y^v$ is the unique anti-canonical curve, the set of nodal roots is finite. Thus, by Harbourne \cite{Harbourne:1985}, it is easy to check if $\omega$ is realizable.
\begin{thm}\label{T:check}
Suppose $0 \in v \cdot \Phi_n$ and $\omega \in W_n^v$. Then $\omega\in Aut (S^v, Y^v \phi^v)$ if and only if there is no nodal root $\alpha \in \Phi(\phi)^+$ such that $\omega^{-1} \alpha \in \Phi^-$.
\end{thm}

\begin{proof}

This is an immediate consequence of Theorem \ref{T:harbourne} and Theorem \ref{T:autgroup}.
\end{proof}

\begin{thm} \label{T:withnodal2}
If $0 \in v \cdot \Phi_n$ and $W_n^v \ne \emptyset$, there is $\omega \in W_n^v \cap W(S^v, \phi^v)$.
\end{thm}

\begin{proof}
Suppose $\omega \in W_n^v$. If $\alpha$ is a nodal root, then by Lemma \ref{L:nodalv} and Proposition \ref{P:periodic}, we see that $\alpha$ is a periodic root. Since $X$ is irreducible and the spectral radius $\lambda_\omega>1$,
from Proposition \ref{P:preservingpositivity}, we know that there is $s \in \Sigma$ such that $(\omega^{-1}  \circ s) \alpha$ is a positive root for every nodal root $\alpha$ where $\Sigma$ is a span of all reflections through nodal roots for $\omega$. In other words, $(\omega^{-1}  \circ s) ( \Phi(\phi)^+) \subset \Phi_n^+$. Thus by Theorem \ref{T:harbourne} and Theorem \ref{T:check}, we see that $ s^{-1} \circ \omega \in W(S^v, \phi^v)$.
Now since any nodal root $\alpha$ satisfies $\alpha \cdot v = 0$, thus $s_\alpha v = v$ for all reflection $s_\alpha$ through a nodal root $\alpha$ and 
\[  s^{-1} \circ \omega (v) =  s^{-1} (\lambda v) = \lambda v.  \]
Thus $ s^{-1} \circ \omega \in W_n^v \cap W(S^v, \phi^v)$.
\end{proof}

Let the Weyl spectrum \[ \Lambda = \cup_n\{ \lambda_\omega : w \in W_n\}\] where $\lambda_\omega$ is the spectral radius of $\omega$. By Lemma \ref{L:essentialdegree}, we see that $\Lambda =  \cup_n\{ \lambda_\omega : w \in W^{ess}_n\}$. Thus we have

\begin{proof}[Proof of Theorem A]
Suppose $\lambda \in \Lambda$. If $\lambda>1$, then there is $\omega\in W^{ess}_n$ for some $n\ge 10$ such that $\lambda_\omega = \lambda$. Let $v$ be the eigenvector of $\omega$ corresponding to $\lambda$.
 By Theorem \ref{T:nonodal} and Theorem \ref{T:withnodal2} we have $\omega' \in \text{Aut}(S^v,Y^v, \phi^v)$ and thus there is an automorphism on $S^v$ such that $\phi^v \circ F_* = \omega' \circ \phi^v$ and $\delta(F) = \lambda$. For $\lambda =1$, since the Cremona involution $s_0 s_1s_2 \in W_3$ lifts to an automorphism $F$ and $\lambda_{s_0 s_1s_2} =1$, we have the desired result. 
\end{proof}

\begin{proof}[Proof of Theorem B]
Suppose an essential element $\omega \in \cup_{n\ge 10} W_n$ with the spectral radius $\lambda= \lambda(\omega)$ and$\omega v= \lambda v$. Then $\omega \in W_n^v$. If no nodal roots are mapped to negative roots by $\omega^{-1}$, by Theorem \ref{T:nonodal} -- \ref{T:check} together with Theorem \ref{T:autgroup}, we see that $\omega$ is realized by an automorphism on $(S^v,Y^v, \phi^v)$. On the other hand, Suppose there is a nodal root $\alpha$ such that  $\omega^{-1} \alpha \in \Phi^-_n$ and let $C_\alpha$ be a projection of $\alpha$ on $S^v$. If $\omega$ is realized by an automorphism $f$ on $(S^v,Y^v, \phi^v)$, then the class $[f^{-1} C_\alpha]= f^{-1}_* \alpha$ is not effective which is impossible. Thus, we have the first assertion. \hfill\break
In case $\omega$ is not realizable, then by Lemmas  \ref{L:nodal0} -- \ref{L:nodaln}, we can find $s \in \Sigma_v$ such that  $s\circ \omega \in W_n^v$ and $s \circ \omega$ maps no nodal roots to negative roots. Notice that this can be done in finitely many steps because there are finitely many nodal roots in $(S^v,Y^v, \phi^v)$. Thus, by Theorems \ref{T:nonodal} -- \ref{T:withnodal2}, we see that $s \circ \omega$ is realizable. 
\end{proof}

\section{Quadratic Birational Maps}\label{S:quad}
Let $(S,Y,\phi) = \text{B}\ell(X, \rho)$ be a marked blowup with the unique anti-canonical curve $Y$ such that $\pi Y= X$. Then by Diller \cite{Diller:2011} we have the following

\begin{thm}\cite{Diller:2011}\label{T:invariant} Suppose $F:S \to S$ is an automorphism stabilizing $Y$ and $\delta(F) >1$, then $X$ is either a cuspidal cubic, a conic with its tangent, or three lines joining at a single point. 
\end{thm}

\begin{proof}
Suppose the automorphism $F$ covers a birational map$f$ on $\mathbb{P}^2$. 
In \cite{Diller:2011}, Diller showed that if $X$ is not one of three kinds in the statement of Theorem, then the determinant, $D(f|_X)$ of $f$, has modulus $1$. However by Theorem \ref{T:determinant}, we have $D(f|_X) = \delta(F) >1.$
\end{proof}

\subsection{Cremona Involution}
Suppose $i,j,k$ are three distinct integers between $1$ and $n$. The Cremona involution $\kappa_{i,j,k} \in W_n$ is given by  the reflection through the vector $\alpha_{i,j,k} = e_0 - e_i-e_j-e_k$ 
\[ \kappa_{i,j,k} (x) =  x + (x \cdot \alpha_{i,j,k} ) \, \alpha_{i,j,k}. \] Recall that the reflection $s_i$ is given by 
\[ s_i(x) = x +( x \cdot \alpha_i) \alpha_i, \ \quad \alpha_i = e_i - e_{i+1}. \] Let us denote $\Sigma_n \subset W_n$ a subspace spanned by reflections $s_i$ for $i=1, \dots, n-1$

\subsection{Quadratic Birational Maps}
Suppose $i,j, k$ are three distinct integers in $\{1, \dots, n\}$ and $s \in \Sigma_n$. And let \[\omega = \kappa_{i,j,k} \circ s \in W_n^{ess}, \qquad  n\ge 10. \] 
Suppose there exists a marked blowup $(S,Y, \phi) = B\ell(X, \rho)$ such that $\omega$ is realized by $F \in Aut(S,Y)$, i.e., $ \omega = F_* $ through the isomorphism $\phi$. Then 
\begin{itemize}
\item there exists a set of (possibly infinitely) points $p_i \in X'$ in the smooth locus of $X$ such that \[\pi \mathcal{E}_i =p_i \qquad \text{where} \qquad\phi(e_i) = [ \mathcal{E}_i], \]
\item there exists a quadratic birational map $f: \p^2 \dasharrow \p^2$ such that \[ \pi \circ F = \pi \circ f, \qquad \text{ and}\] 
\item the determinant of $f|_X = \delta(F) = \lambda$ where $\lambda$ is the spectral radius of $\omega>1$, i.e. \[ f|_X(t) = \lambda t +b,\qquad t \in \mathbb{C}.\]
\end{itemize}
Since $\omega$ is a composition of exactly one Cremona involution and a reflection in $\Sigma_n$, both $f$ and $f^{-1}$ have exactly three distinct (but possibly infinitely near) points of indeterminacy. 
\begin{lem}
The sets of indeterminacy points for $f$ and $f^{-1}$ are\[ \text{Ind}(f) = \{ p_{i'}, p_{j'},p_{k'}\}, \qquad \text{and} \qquad \text{Ind}(f^{-1}) =   \{ p_i, p_j,p_k\}\] where $s(e_{i'}) = e_i, s(e_{j'}) = e_j, s(e_{k'}) = e_k$.
\end{lem}

\begin{proof}
Since $F$ is an automorphism if  $p \in \mathbb{P}^2$ is a point of indeterminacy for $f$, $p$ must be a base point for the marked blowup $(S, Y, \phi) = B\ell(X, \rho)$. 
Suppose $p= p_\ell$ for some $\ell = 1, \dots, n$. Since $F_* e_\ell = \omega e_\ell$ and $F_* e_\ell \cdot e_0 \ne 0$ if $s(e_\ell) \in \{ e_i, e_j, e_k \}$. Thus we have $\text{Ind}(f) = \{ p_{i'}, p_{j'},p_{k'}\}$ with $s(e_{i'}) = e_i, s(e_{j'}) = e_j, s(e_{k'}) = e_k$. With the similar argument, we have $\text{Ind}(f^{-1}) =   \{ p_i, p_j,p_k\}$.
\end{proof}

\subsection{Orbit Data} \label{SS:orbitdata}

Since $\omega \in W_n^{ess}$ for each $\ell = 1, \dots, n$, there is a smallest positive integer $n_\ell$ such that $\omega^{n_\ell} e_\ell \cdot e_0 >0 $, in other words \[ \omega^{n_\ell-1} e_\ell \in \{ e_{i'}, e_{j'}, e_{k'} \}. \]

\begin{lem}\label{L:perm}
There is a permutation $\sigma$ of $\{i,j,k\}$ such that \[ \omega^{n_s} e_s = e_{\sigma(s)'}, \qquad s \in \{i,j,k\}. \] 
\end{lem}

\begin{proof}
It is sufficient to show that $\omega^{n_i} e_i \ne \omega^{n_j} e_j$ for $i \ne j$. Suppose $\omega^{n_i} e_i = \omega^{n_j} e_j$ and $n_i \le n_j$. First we have \[ -1 = \omega^{n_i} e_i  \cdot  \omega^{n_j} e_j = e_i \cdot  \omega^{n_j-n_i} e_j \quad \Rightarrow  \quad   \omega^{n_j-n_i} e_j= e_i \quad \Rightarrow \quad n_j - n_i \gneqq 0. \] Then, we have \[  \omega^{n_j-n_i -1} e_j = \omega^{-1} e_i = s^{-1}( e_0 - e_j-e_k)\quad  \Rightarrow  \quad  \omega^{n_j-n_i -1} e_j \cdot e_0 \ne 0.\] Since $n_j$ is the smallest positive integer such that $\omega^{n_\ell} e_\ell \cdot e_0 >0 $ and since $e_j \cdot e_0 =0$, we have \[ n_j-n_i -1 \le 0 , \qquad \text{ and } n_j-n_i -1 \ne 0.\]
It follows that $ n_j-n_i  \leq 0$.
\end{proof}

\begin{lem}\label{L:olength}
The sum $n_i + n_j + n_k = n$. Furthermore, the sets $O_i=\{ \omega^s e_i, 0 \le s \le n_i-1\},O_j\{ \omega^s e_j, 0 \le s \le n_j-1\}$ and $O_k=\{ \omega^s e_k, 0 \le s \le n_k-1\}$ are pairwise disjoint.
\end{lem}
\begin{proof}
It is sufficient to show that $O_i, O_j, O_k$ are pairwise disjoint. Suppose there are $0 \le s \le n_i-1$ and $0\le t \le n_j-1$ such that $\omega^s e_i = \omega^t e_j$. We may assume $s<t$. It follows that $\omega^{t-s} e_j = e_i$. Since $n_i $ is the smallest positive integer such that $\omega^n_i e_i\cdot e_0 >0$ and since $t-s < n_j$, we have $n_j = t-s+ n_i$ and $\omega^{n_j -1} e_j = \omega ^{n_i-1} e_i$. This is a contradiction because of the previous Lemma \ref{L:perm}.
\end{proof}

We can rewrite the above Lemma in terms of the birational map $f$ associated with $\omega$. 
\begin{lem}\label{L:orbitdata}
There are three positive integers $n_i, n_j,n_k$ and a permutation $\sigma$ on $\{i,j,k\}$ such that for each $\ell \in \{i,j,k \}$
\[ f^{n_\ell-1} (p_\ell) = p_{\sigma(\ell)'}, \qquad   f^{s} (p_\ell) \notin \text{Ind}(f) \ \ \text{for all } \ \ 0 \le s \le n_\ell-2. \]
\end{lem} 

We call these numerical data $n_i, n_j,n_k$ and a permutation $\sigma$ on $\{i,j,k\}$ defined in Lemma \ref{L:orbitdata} \textit{the orbit data} of $f$. Also since \[ \omega : e_0 - e_{i'}-e_{j'} \mapsto (e_0 - e_i-e_j-e_k) - (e_0 - e_j-e_k) - (e_0 - e_i-e_k) = e_k,\] we see that

\begin{equation*}
f \  : \ \left\{ \begin{aligned} &E_i \to p_i \to  f(p_i) \to \cdots f^{n_i-1} (p_i) = p_{\sigma (i)'} \\&E_j \to p_j \to  f(p_j) \to \cdots f^{n_j-1} (p_j) = p_{\sigma(j)'}\\&E_k \to p_k \to  f(p_k) \to \cdots f^{n_k-1} (p_k) = p_{\sigma(k)'} \end{aligned} \right.
\end{equation*}
where $E_i$ is the line joining $e_{j'}$ and $e_{k'}$, $E_j$ is the line joining $e_{i'}$ and $e_{k'}$, and $E_k$ is the line joining $e_{i'}$ and $e_{j'}$.

\subsection{Fixing a Cubic}

For any curve $C \in \p^2$, we define \[ f(C) = \overline{ f(C \setminus \text{Ind}(f) }. \] The degree of $f(C)$ is determined by $F_*(\tilde C) \cdot e_0$ where $\tilde C$ is the strict transform of $C$ in $S$. Thus we have \[ \text{deg} f(C) =  2 \text{deg}(C) - \sum{p \in \text{Ind}(f)} \nu_p(C)\] where $\nu_p(C)$ is multiplicity of $C$ at $p$. Using this, Diller \cite{Diller:2011} gives sufficient conditions on the orbit data for $f$ to preserve the projection $X$ of the anti-canonical curve $Y$. 

\begin{thm}\cite[Theorem~4.1]{Diller:2011}\label{T:3lines}
Suppose $X$ is three lines $L_1, L_2, L_3$ joining at a single point, and $X$ is preserved by $f$. Then, each line contains exactly one point of indeterminacy for $f$ and exactly one point of indeterminacy for $f^{-1}$. Furthermore, we have one of the following
\begin{enumerate}
\item $\sigma$ is an identity permutation,  
\item $\sigma = (i\, j)$ is a transposition exchanging $i$ and $j$ and both $n_i$ and $n_j$ are odd, 
\item $\sigma$ is a cyclic permutation, then $n_i \equiv n_j \equiv n_k \equiv 1 \text{ or } 2$ $(\text{mod}\ 3)$.
\end{enumerate}
\end{thm}

\begin{proof} To make the notation simple, let us assume that $\{i,j,k\} = \{1,2,3\}$. Since each irreducible component of $X$ is a line, $\text{deg} f^\pm(L_i) =1$ for all $i=1,2,3$. It follows that each line contains exactly one point of indeterminacy for $f^\pm$. It follows that all three points of indeterminacy are distinct. Suppose $p_{i'} \in L_i$ for $i=1,2,3$. Thus, for each $i=1,2,3$, only $L_i$ intersects $E_i$ at a non-indeterminate point, say $q_i$. Since $f|_X$ is an automorphism, we have $f|_X(q_i) = p_i$. 
We have three possibilities:
\begin{itemize}
\item \textit{Case 1: $f(L_i) = L_i, i=1,2,3$.} Since $q_i, p_{i'}  \in L_i$, the permutation $\sigma = Id$.
\item  \textit{Case 2: $f: L_1 \leftrightarrow L_2$} Since $q_1 \in L_1$, we have $p_1 \in L_2$. Since $p_{2'} \in L_2$. Thus if $\sigma= (1 \, 2)$ then $n_1$ has to be odd and so does $n_2$ and if $\sigma=Id$ then both $n_1$ and $n_2$ are even.
\item \textit{Case 3: $f$ permutes $L_i$ cycliclly} Suppose $f:L_1 \to L_2 \to L_3 \to L_1$. Similar to case 2, $p_1 \in L_2$. Since $p_{i'} \in L_i, i=1,2,3$ we have 
\begin{enumerate}
\item if $\sigma(1) = 1$ then $n_1 \equiv 0 $ (mod $3$),
\item if $\sigma(1) = 2$ then $n_1 \equiv 1 $ (mod $3$), and
\item if $\sigma(1) = 3$ then $n_1 \equiv 2 $ (mod $3$).
\end{enumerate}
Thus if $\sigma$ is a cyclic permutation then $n_i \equiv n_j \equiv n_k \equiv 1 \text{ or } 2$ $(\text{mod}\ 3)$,  if $\sigma=Id$ then $n_i \equiv n_j \equiv n_k \equiv 0$ $(\text{mod}\ 3)$, and if $\sigma = (1 \, 2)$ then $n_3 \equiv 0$ (mod $3$) and $n_1 \equiv n_2 \equiv 1 \text{ or } 2$ $(\text{mod}\ 3)$.
\end{itemize}
The assertion follows from the above three cases. The detailed proof is also available in \cite{Diller:2011, Kim:2022}.
\end{proof}

Also, we have 
\begin{thm}\label{T:conicT}
If $X$ is a union of a conic $Q$ and its tangent $L$ and $X$ is preserved by $f$, then exactly one of the following occurs
\begin{enumerate}
\item there exist a unique $\ell \in \{ i,j,k\}$ such that $p_\ell, p_{\ell'} \in L$ and $\sigma(\ell) = \ell$
\item three positive integers $n_i,n_j,n_k$ are all odd, and every point of indeterminacy belongs to $Q \setminus L$ \[ \text{Ind}(f) \cup \text{Ind}(f^{-1}) \subset Q \setminus L. \]
\end{enumerate}
\end{thm}

\begin{proof} Again, let us assume $\{i,j,k\} = \{1,2,3\}$. We have either $f(L) = Q$ or $f(L) = L$. If $f(L) = L$, then $L$ contains exactly one point of indeterminacy, say $p_{1'}$. Then, using the same notations in the proof of the previous Theorem, we have both $q_1=L_1 \cap X, p_{1'} \in L$, and so does $p_1$ and $\sigma(1) = 1$. 

If $f(L) = Q$, then $L$ does not contain any indeterminant points. Also, any exceptional line $E_i$ intersects $Q$ at two indeterminant points. So there is a non-indeterminate intersection point $q_i = E_i \cap L$. It follows that $f(q_i) = p_i \in Q$ for all $I=1,2,3$. Since all points of indeterminacy belongs to $Q$ and $f$ interchange $Q$ and $L$, we see that $f^{even} q_i \in L$ and $f^{odd} q_i \in Q$ and all $n_i$ has to be odd. 
\end{proof}

\section{UnRealizable essential elements}\label{S:trouble}
Let us consider $\omega = \kappa_{1,7,8} \circ s  \in W_{13}$ such that the Cremona involution $\kappa_{1,7,8}$ is defined by 
\[  \begin{aligned} \ \kappa_{1,7,8} \ : \ & e_0  \mapsto 2 e_0 - e_1-e_7-e_8 \\& e_1 \mapsto e_0 - e_7-e_8\\& e_7 \mapsto e_0 - e_1-e_8\\& e_8 \mapsto e_0 - e_1-e_7\\ & e_i \mapsto e_i \qquad \text{ otherwise } \end{aligned} \]
and a cyclic permutation $s$ is given by 
\[ s : e_0 \mapsto e_0, \qquad\text{ and } \qquad  e_i \mapsto e_{i+1} \ \ \ i=1, \dots, 12, \qquad\text{ and } \qquad e_{13} \mapsto e_1. \]  
It follows that 
\begin{equation*}
\omega \ :  \ \left\{ \begin{aligned} & e_1 \to e_2 \to e_3 \to e_4 \to e_5 \to e_6 \to e_0-e_1-e_8\\ & e_7 \to e_0 - e_1 - e_7\\ & e_8 \to e_9 \to e_{10} \to e_{11} \to e_{12} \to e_{13}  \to e_0-e_7-e_8\\
\end{aligned} \right.
\end{equation*}

Thus if $\omega$ is realized by an automorphism $F$ on a rational surface $S$ covers a birational map $f$, then we have \[  \text{Ind}(f^{-1} )= \{ p_1, p_7, p_8\} \qquad \text{and} \qquad \text{Ind}(f)= \{p_{7'}=p_6, p_{8'}=p_7, p_{1'}=p_{13} \}, \] and hence the orbit data of $f$ is given by $6,1,6$ and a cyclic permutation. 
The direct computation shows that the spectral radius $\lambda(\omega)$ is given by the largest real zero of $t^6-t^5-t^3-t+1$. In summary, we have

\begin{lem}
If $\omega$ is realized by an automorphism $F$ on some rational surface $S$, then $F$ covers a quadratic birational map $f$ with orbit data $6,1,6$ with a cyclic permutation. Furthermore the determinant of $f  = \lambda(\omega) \sim 1.50614$, i.e. the restriction of $f$ to the smooth locus of $X$ is given by $f|_X(t) = \lambda(\omega) t + b$, $t \in \mathbb{C}$ for some $b \in \mathbb{C}$.
\end{lem}

\begin{prop}\label{P:ccubic}
$\omega$ can not be realized by an automorphism $F$ on an anti-canonical rational surface $(S, Y)$ such that $X=\pi(Y)$ is a cuspidal cubic.
\end{prop}

\begin{proof}
Suppose $\omega$ is realized by an automorphism $F$ on an anti-canonical rational surface $(S, Y)$. Then $F$ covers a birational map $f$ fixing $X$. 
By Diller \cite[Theorem~1.3]{Diller:2011}, $f$ is uniquely determined up to linear conjugacy. We may assume $f$ is the one constructed in Section \ref{S:realization}. 
From direct computation, we see that an eigenvector $v = [ v_0, v_1, \dots, v_{13} ] $ of $\omega$ for the spectral raidius is given by 
\[  v_2=v_8, \ v_3= v_9, \ v_4=v_{10},\  v_5=v_{11},\  v_6=v_{12}, \qquad \text{ and }\] \[ \{ v_1, v_2, v_3, v_4, v_5, v_6,v_7, v_{13}\} \text{ is a set of } 8 \text{ distinct numbers.} \]
It follows that the base locus $\{p_1, \dots, p_{11}\}$ of the marked blowup satisfying $p_{i+6}$ is infinitely near point to $p_i$ of order $1$ for $i=2, \dots, 6$. 
Also by inspection, we see that there are $6$ nodal roots for the marked blowup $(S^v, Y^v, \phi^v)$:
 \[ \Phi(\phi^v)^+ = \{ e_2-e_8, e_3-e_9, e_4-e_{10}, e_5-e_{11}, e_6-e_{12}, e_0 - e_1-e_7-e_{13} \} \]
and \[ \omega (e_0 - e_1-e_7-e_{13}) = - (e_2-e_8). \]
Thus, $\omega^{-1}$ does not preserve the positivity of the effective roots, and by Theorem \ref{T:harbourne} we see that $\omega \notin Aut(S^v, Y^v,\phi^v)$. 
\end{proof}

\begin{prop}\label{P:others}
$\omega$ can not be realized by an automorphism $F$ on an anti-canonical rational surface $(S, Y)$ such that $X=\pi(Y)$ is either three lines joining at a single point or a conic with its tangent. 
\end{prop}

\begin{proof}
Suppose $f$ is a birational map fixing $X$, and $F$ covers $f$.
If $X$ is three lines joining at a single point, by Theorem \ref{T:3lines}, there is no such $f$ since $\sigma$ is a cyclic permutation and $1 \not\equiv 6 $ (mod $3$).
If $X$ is a conic and its tangent, by Theorem \ref{T:conicT}, there is no such $f$ since  $\sigma$ is a cyclic permutation, and $6$ is not odd. 
\end{proof}

\begin{thm}\label{T:nono}
$\omega$ can not be realized by an automorphism $F$ on an anti-canonical rational surface $(S, Y)$.
\end{thm}

\begin{proof}
By Theorem \ref{T:invariant}, the only possibility for $X$ is either a cuspidal cubic, a conic and its tangent, or three lines joining at a single point. The result follows from the previous two Propositions \ref{P:ccubic} and \ref{P:others}.
\end{proof}

\begin{proof}[Proof of Theorem C]
The existence of an essential element that cannot be realized by an automorphism on an anti-canonicial rational surface is clear from the previous Theorem \ref{T:nono}
\end{proof}

\begin{thm}\label{T:yes}
Let $s'$ be a reflection through a vector $e_2-e_{8}$. Then $s' \circ \omega$ is realized by an automorphism on an anticanonical rational surface $(S, Y)$ such that $Y$ is the strict transform of a cuspidal cubic in $\mathbb{P}^2$.
\end{thm}

\begin{proof}
Notice that $(e_2-e_{8}) \cdot v = 0 $. Thus $\omega$ and $s' \circ \omega$ have the same spectral radius, and an eigenvector $v$ for $\omega$ associated the spectral radius $\lambda(\omega)$ also satisfies
\[ s' \circ \omega (v) = \lambda(\omega) v.\]
Furthermore, we have
\[ s' \circ \omega : e_2-e_8 \to e_3-e_9 \to e_4-e_{10} \to e_5-e_{11} \to e_6-e_{12} \to e_0 -e_1-e_7-e_{13} \to e_2-e_8. \]
Thus $(s' \circ\omega)^{-1}$ preserves the positivity of nodal roots. By Theorem \ref{T:harbourne}, we see that $s'\circ \omega \in (S^v, Y^v, \phi^v)$.
\end{proof}




\bibliographystyle{plain}
\bibliography{biblio}

 \end{document}